\documentclass[preprint,authoryear,12pt]{paper}
\usepackage{amsmath,amssymb,amsthm,enumerate,color,bbm,ifthen,graphicx}
\usepackage[latin1]{inputenc}

\newcounter{hypC}
\newenvironment{hypC}{\refstepcounter{hypC}\begin{itemize}
  \item[{\bf C\arabic{hypC}}]}{\end{itemize}}

\newcounter{hypA}
\newenvironment{hypA}{\refstepcounter{hypA}\begin{itemize}
  \item[{\bf A\arabic{hypA}}]}{\end{itemize}}

\newcounter{hypAver}
\newenvironment{hypAver}{\refstepcounter{hypAver}\begin{itemize}
  \item[{\bf AVER\arabic{hypAver}}]}{\end{itemize}}

\newtheorem{theo}{Theorem}[section]
\newtheorem{lemma}[theo]{Lemma}

\newtheorem{prop}[theo]{Proposition}
\newtheorem{remark}[theo]{Remark}

\newcommand{\eqdef}{\ensuremath{\stackrel{\mathrm{def}}{=}}}
\def\eqsp{\;}

\def\Rset{\mathbb R}

\newcommand{\plim}{\ensuremath{\stackrel{\mathbb{P}}{\longrightarrow}}}

\newcommand{\aslim}{\ensuremath{\stackrel{\text{a.s.}}{\longrightarrow}}}

\def\PE{\mathbb E}
\def\PP{\mathbb P}
\def\F{\mathcal{F}}
\def\A{\mathcal{A}}
\def\K{\mathcal{K}}

\newcommand{\un}{\ensuremath{\mathbbm{1}}}

\def\Id{\mathrm{Id}}

\def\btheta{\bar{\theta}}

\def\hH{\widehat{H}}

\def\etal{et al.}
\def\ie{i.e.}

\begin{document}

\title{ Central Limit Theorems for Stochastic Approximation  \\
  with controlled Markov chain dynamics}
\author{Gersende Fort \\
  \small{ LTCI, CNRS \& TELECOM ParisTech,} \\
  \small{46 rue Barrault, 75634 Paris Cedex
  13, France. }  \\
  \small{gersende.fort@telecom-paristech.fr }}

\date{}

\maketitle


\begin{abstract}
  This paper provides a Central Limit Theorem (CLT) for a process $\{\theta_n,
  n\geq 0\}$ satisfying a stochastic approximation (SA) equation of the form
  $\theta_{n+1} = \theta_n + \gamma_{n+1} H(\theta_n,X_{n+1})$; a CLT for the
  associated average sequence is also established.  The originality of this
  paper is to address the case of controlled Markov chain dynamics $\{X_n,
  n\geq 0 \}$ and the case of multiple targets. The framework also accomodates
  (randomly) truncated SA algorithms.
  
  Sufficient conditions for CLT's to hold are provided as well as comments on
  how these conditions extend previous works (such as independent and
  identically distributed dynamics, the Robbins-Monro dynamic or the single
  target case). The paper gives a special emphasis on how these conditions hold
  for SA with controlled Markov chain dynamics and multiple targets; it is
  proved that this paper improves on existing works.
\end{abstract}

\paragraph{Acknowledgements.}
I gratefully acknowledge Prof. P. Priouret for fruitful discussions.

\newpage
\tableofcontents

\clearpage
\newpage

\section{Introduction} Stochastic Approximation (SA)
algorithms were introduced for finding roots of an unknown function $h$ (for
recent surveys on SA, see e.g.
\cite{chen:2002,spall:2003,kushner:yin:2003,borkar:2008,kushner:2010}).  SA
defines iteratively a sequence $\{\theta_n, n\geq 0\}$ by the update rule
\begin{equation}
  \label{eq:SAgeneral}
  \theta_{n+1} = \theta_n + \gamma_{n+1} \Xi_{n+1} \eqsp,
\end{equation}
where $\{\gamma_{n}, n\geq 1 \}$ is a sequence of deterministic step-size and
$\Xi_{n+1}$ is a random variable (r.v.) standing for a noisy measurement of the
unknown quantity $h(\theta_n)$.

\medskip

Our aim is to establish the rate of convergence of the sequence $\{\theta_n, n
\geq 0 \}$ to a limiting point $\theta_\star$ in the following framework.

Let $\Theta \subseteq \Rset^d$; the sequence $\{\theta_n, n\geq 0 \}$ is a
$\Theta$-valued random sequence defined on the filtered probability space
$(\Omega, \mathcal{A}, \PP, \{\F_n, n\geq 0 \})$ and given by
\[
\theta_{n+1} = \theta_n + \gamma_{n+1} \left(h(\theta_n) + e_{n+1} +r_{n+1}
\right) \eqsp, \qquad \theta_0 \in \Theta \eqsp;
\]
where $h: \Theta \to \Rset^d$ is a measurable function, $\{e_n, n\geq 1 \}$ is
a $\F_n$-adapted $\PP$-martingale increment sequence and $\{r_n, n \geq 1 \}$
is a vanishing $\F_n$-adapted random sequence. Such a general description
covers many SA algorithms: as discussed below (see
Section~\ref{sec:SAtheory:assumptions}), it covers the case when $\Xi_{n+1}$ is
of the form $H(\theta_n, X_{n+1})$ where $\{X_n,n \geq 1 \}$ are independent
and identically distributed (i.i.d.) r.v. such that (s.t.) $\PE\left[ H(\theta,
  X) \right] = h(\theta)$; and the more general case when $\{X_n, n\geq 1 \}$
is an adapted (non stationary) Markov chain with transition kernel driven by
the current value of the SA sequence $\{\theta_n, n\geq 0\}$. It also covers
the case of fixed truncated and randomly truncated SA algorithms \ie
situations when given a (possibly random) sequence of subsets $\{\K_n, n \geq 0
\}$ of $\Theta$, the update rule is given by
\begin{equation}
  \label{eq:truncatedSA:general}
  \theta_{n+1} = \left\{ 
  \begin{array}{ll}
\theta_n + \gamma_{n+1} \Xi_{n+1}\eqsp, & \ \text{if $\theta_n  + \gamma_{n+1} \Xi_{n+1} \in \K_{n+1}$} \\
\theta_0 & \ \text{otherwise \eqsp.}
  \end{array}
\right.
\end{equation}
Such a truncated algorithm is used for example to solve optimization problem on a
constraint set $\Theta$ (in this case, $\K_n = \Theta$ for any $n$), or to
ensure stability of the random sequence $\{\theta_n, n\geq 0 \}$ in situations
where the location of the sought-for root is unknown (in this case, $\K_n$ is
an increasing sequence of sets, see~\cite{chen:guo:gao:1988} and \cite[Chapter
2]{chen:2002}).

Our second aim is to extend the previous results to the case of multiple
targets: we provide asymptotic convergence rates of $\{\theta_n, n\geq 0\}$ to
a point $\theta_\star$ given the event $\{\lim_q \theta_q = \theta_\star \}$
for some $\theta_\star$ in the interior of $\Theta$. Note that this paper is
devoted to convergence rates so that sufficient conditions for the convergence
is out of the scope of the paper; for convergence, the interested reader can
refer to
\cite{benveniste:metivier:priouret:1987,duflo:1997,benaim:cours99,chen:2002,andrieu:moulines:priouret:2005,borkar:2008}.

\medskip 

The originality of this paper consists in deriving rates of convergence in a
new framework characterized by \textit{(i)} general assumptions on the noisy
measurement $\Xi_{n+1}$ of $h(\theta_n)$ which weaken the conditions in the
literature and \textit{(ii)} the multiple targets problem.  In
Section~\ref{sec:SAtheory:comment}, our framework will be carefully compared to
the literature.

We derive sufficient conditions on the step-size sequence $\{\gamma_n, n \geq 1
\}$, on the random sequences $\{e_n, r_n, n \geq 1 \}$ and on the limiting
point $\theta_\star$ so that $\gamma_n^{-1/2} (\theta_n - \theta_\star)$
converges in distribution under the conditional probability $\PP(\cdot \vert
\lim_q \theta_q = \theta_\star)$. The limiting distribution is a (mixture of)
centered Gaussian distribution(s) and this distribution is explicitly
characterized.  We also address the rate of convergence of the associated
averaged process $\{\btheta_n, n\geq 0 \}$ defined by
\begin{equation}
  \label{eq:averaged:sequence:def}
 \btheta_n \eqdef \frac{1}{n+1} \sum_{k=0}^{n} \theta_k \eqsp.
 \end{equation}
 We prove that this averaged sequence reaches the optimal rate and the optimal
 variance (in a sense discussed below); such a result was already established
 in the literature in a more restrictive framework.

 The paper is organized as follows. Section~\ref{sec:SAtheory} (resp.
 Section~\ref{sec:averSAtheory}) is devoted to the SA sequence $\{\theta_n,
 n\geq 0\}$ (resp. the averaged SA sequence $\{\btheta_n, n\geq 0 \}$). We
 successively introduce the assumptions, comment these conditions, compare our
 framework to the literature and state a Central Limit Theorem (CLT). In
 Section~\ref{sec:appli}, our results are applied to a randomly truncated SA
 algorithm with controlled Markov chain dynamics; since our conditions are
 quite weak, we are able to obtain better convergence rates than the rates
 obtained in Delyon~\cite{delyon:2000}.  All the proofs are postponed in
 Section~\ref{sec:proofs}.

\section{A Central Limit Theorem for Stochastic Approximation}
\label{sec:SAtheory}
\subsection{Assumptions} 
\label{sec:SAtheory:assumptions}
Let $\Theta \subseteq \Rset^d$.  We consider the $\Rset^d$-valued sequence
satisfying for $n \geq 0$,
\begin{equation}
  \label{eq:definition:theta}
  \theta_{n+1}=\theta_{n}+\gamma_{n+1}h(\theta_{n})+\gamma_{n+1}e_{n+1}+\gamma_{n+1}r_{n+1}
\eqsp, \qquad \qquad \theta_0 \in \Theta \eqsp;
\end{equation}
and we establish a Central Limit Theorem along sequences $\{\theta_n, n\geq
0\}$ converging to some point $\theta_\star \in \Theta$ which is a root of the
function $h$. We assume the following conditions on the attractive target
$\theta_\star$.
\begin{hypC}
\label{hyp:C1}
  \begin{enumerate}[(a)]
  \item \label{hyp:C1:a} $\theta_{\star}$ is in the interior of $\Theta$ and
    $h(\theta_{\star})=0$.
  \item \label{hyp:C1:b} The mean field $h:\Theta\to \Rset^{d}$ is measurable
    and twice continuously differentiable in a neighborhood of $\theta_\star$.
  \item \label{hyp:C1:c} The gradient $\nabla h(\theta_{\star})$ is a Hurwitz
    matrix. Denote by $-L$, $L>0$, the largest real part of its eigenvalues.
\end{enumerate}
\end{hypC}
Let $\{e_n, n\geq 1\}$ be a $\Rset^d$-valued random variables defined on the
filtered space $(\Omega, \mathcal{A},\PP, \{\F_n, n\geq 0\})$. We will denote
by $|\cdot|$ the Euclidean norm on $\Rset^d$; and by $x^T$ the transpose of a
matrix $x$. By convention, vectors are column-vectors. For a set $A$, $\un_A$
is the indicator function. It is assumed
\begin{hypC}\label{hyp:C3}
 \begin{enumerate}[(a)]
 \item \label{hyp:C3:a} \ $\{e_{n},\ n\geq 1\}$ is a $\F_n$-adapted
   $\PP$-martingale-increment sequence \ie $\PE\left[e_n \vert \F_{n-1}
   \right] = 0$ $\PP$-almost surely.
 \item \label{hyp:C3:b} For any $m \geq 1$, there exists a sequence of
   measurable sets $\{\A_{m,k}, k \geq 0 \}$ such that $\A_{m,k} \in \F_k$ and
   there exists $\tau >0$ such that 
    \[\sup_{k \geq 0} \PE \left[|e_{k+1}
      |^{2+\tau}\un_{\A_{m,k}} \right]<\infty \eqsp.
\]
In addition, for any $m \geq 1$, $\lim_k \un_{\A_{m,k}} \un_{\lim_q \theta_q =
  \theta_\star} = \un_{\A_m} \un_{\lim_q \theta_q = \theta_\star}$ and the
limiting set satisfies $\lim_m \PP(\A_m \vert \lim_q \theta_q = \theta_\star)
=1$.
\item \label{hyp:C3:c} $\PE \left[e_{k+1}e_{k+1}^{T}|\F_{k} \right] = U_\star +
  D_{k}^{(1)} + D_{k}^{(2)}$ where $U_\star$ is a symmetric positive
  definite (random) matrix and
 \begin{equation}
      \label{eq:var3}
   \left\{ \begin{array}{l}
D_k^{(1)} \aslim 0  \eqsp,  \qquad \ \text{on the set $\{\lim_q \theta_q =
  \theta_\star \}$}  \\
\lim_n \gamma_n \ \PE\left[ \left| \sum_{k=1}^n D_k^{(2)}  \right|  \un_{\lim_q \theta_q = \theta_\star} \un_{\A_m}\right]  = 0 \eqsp;
\end{array}
\right. \end{equation}
the sequence $\{\A_m, m\geq 1 \}$ is defined  in C\ref{hyp:C3}\ref{hyp:C3:b}.
 \end{enumerate}
\end{hypC} We will show (see remark~\ref{rem:relaxedD2} in Section~\ref{sec:proofs}) that
the condition on the r.v. $\{D_k^{(2)}, k\geq 1 \}$ can be replaced with:
$D_k^{(2)} = D_k^{(2,a)} + D_k^{(2,b)}$
 \begin{equation}
   \label{eq:NewAssumption:C2C}
   \lim_n \gamma_n \ \PE\left[ \left| \sum_{k=1}^n D_k^{(2,a)} \un_{\A_{m,k}}
    \un_{A_k} \right| +  \left| \sum_{k=1}^n D_k^{(2,b)}  \right|  \un_{\A_m}  \un_{\lim_q \theta_q = \theta_\star} \right] = 0 \eqsp, \quad \forall m \geq 1 \eqsp,
 \end{equation}
where $\{A_k, k \geq 1 \}$ is any $\F_k$-adapted sequence of sets satisfying
$\lim_k \un_{A_k} = \un_{\lim_q \theta_q = \theta_\star}$; and $\A_{m,k}$ is
given by C\ref{hyp:C3}\ref{hyp:C3:b}.

For a sequence of $\Rset^d$-valued r.v. $\{Z_n, n\geq 0\}$, we write $Z_n =
O_{w.p.1.}(1)$ if
$\sup_n |Z_n| < \infty$ w.p.1; 
and $Z_n = o_{L^p}(1)$ if $\lim_n \PE[|Z_n|^p] =0$. Let $\{r_n, n\geq 1\}$ be a
$\Rset^d$-valued random variables defined on the filtered space $(\Omega,
\mathcal{A},\PP, \{\F_n, n\geq 0\})$.
\begin{hypC}\label{hyp:C4}  $r_n$ is $\F_n$-adapted. $r_n = r_n^{(1)} + r_n^{(2)}$  with,   for any $m \geq 1$,
  \[
  \left\{
   \begin{array}{l}
\gamma_{n}^{-1/2} r_{n}^{(1)} \ \ \un_{\lim_{q}\theta_{q} =\theta_{\star}}  \un_{\A_m} =
  O_{w.p.1}(1) o_{L^1}(1) \eqsp, \\
\sqrt{\gamma_n} \sum_{k=1}^n r_k^{(2)} \ \ \un_{\lim_q \theta_q = \theta_\star} \un_{\A_m}= O_{w.p.1}(1) o_{L^1}(1) \eqsp.
   \end{array}
\right.
\]
The sequence $\{\A_m, m\geq 1 \}$ is defined in
C\ref{hyp:C3}\ref{hyp:C3:b}.\end{hypC} The last assumption is on the step-size
sequence.
\begin{hypC}\label{hyp:C2}  
 One of the following conditions is satisfied:
   \begin{enumerate}[(a)]
   \item \label{hyp:C2:a:1} $\sum_{k}\gamma_{k}= + \infty,
     \sum_{k}\gamma_{k}^{2}<\infty$ and
     $\log(\gamma_{k-1}/\gamma_{k})=o(\gamma_{k})$.
  \item \label{hyp:C2:a:2} $\sum_{k}\gamma_{k}= + \infty,
    \sum_{k}\gamma_{k}^{2}<\infty$ and there exists $\gamma_{\star}>1/(2L)$
    such that $\log(\gamma_{k-1}/\gamma_{k})\sim\gamma_{k}/\gamma_{\star}$.
  \end{enumerate}
\end{hypC}

\subsection{Comments on the assumptions}
\label{sec:SAtheory:comment}
The framework described by (\ref{eq:definition:theta}) and the conditions
C\ref{hyp:C1} to C\ref{hyp:C2} is general enough to cover many scenarios studied
in the literature and to address new ones. 

For SA algorithms (\ref{eq:SAgeneral}) with $\Xi_{n+1} = H(\theta_n, X_{n+1})$,
$\{X_n, n \geq 1 \}$ i.i.d. r.v. (and independent of $\theta_0$) such that
$h(\theta)= \PE\left[H(\theta, X) \right]$, Eq.~(\ref{eq:definition:theta}) is satisfied with
\[
e_{n+1}= H(\theta_n, X_{n+1}) -h(\theta_n) \eqsp, \qquad \qquad r_{n+1} = 0
\eqsp;
\]
and $\PE\left[e_{n+1} \vert \F_n \right] = 0$. Our framework also addresses the
case when $\{X_n, n\geq 1 \}$ is a $\F_n$-adapted controlled Markov chain \ie
when there exists a family of transition kernels $\{Q_\theta, \theta \in \Theta
\}$ such that
\[
\PP(X_{n+1} \in \cdot \vert \F_n) = Q_{\theta_n}(X_n, \cdot) \eqsp,
\]
each kernel possessing an invariant probability distribution $\pi_\theta$ and
$h(\theta) = \int H(\theta,x) \, \pi_\theta(dx)$ - hereafter, these algorithms
will be called ``SA with controlled Markov chain dynamics''. Introduce the
solution $\widehat{H}_\theta$ of the Poisson equation $H(\theta, \cdot) - h(\theta)
= \widehat{H}_\theta - Q_\theta \widehat{H}_\theta$ (see e.g. \cite[Chapter
8]{hernandez:lasserre:2003} or \cite[Chapter 17]{meyn:tweedie:2009}), and set
\[
e_{n+1} = \widehat{H}_{\theta_n}(X_{n+1}) - Q_{\theta_n} \widehat{H}_{\theta_n}(X_n)
\eqsp, \qquad r_{n+1} = Q_{\theta_n} \widehat{H}_{\theta_n}(X_n) -
Q_{\theta_n} \widehat{H}_{\theta_n}(X_{n+1}) \eqsp;
\]
then $\PE\left[e_{n+1} \vert \F_n \right] =0$ $\PP$-almost surely. We will
provide in Section~\ref{sec:appli} sufficient conditions on the transition
kernels $Q_\theta$ so that these sequences $\{e_n,r_n, n \geq 1 \}$ exist and
satisfy the conditions C\ref{hyp:C3} and C\ref{hyp:C4}. Note that the i.i.d.
case is a special case of the controlled Markov chain framework (set $Q_\theta
= \pi_\theta = \pi$ for any $\theta$); and the so-called Robbins-Monro case
corresponds to $Q_\theta = \pi_\theta$ for any $\theta$.

Truncated SA algorithms~(\ref{eq:truncatedSA:general}) can be written as
\[
\theta_{n+1} = \theta_n + \gamma_{n+1} \Xi_{n+1} + \left( \theta_0 - \theta_n -
  \gamma_{n+1} \Xi_{n+1} \right) \un_{ \theta_n + \gamma_{n+1} \Xi_{n+1} \notin
  \K_{n+1}} \eqsp;
\]
in most (if not any) proof of convergence of this sequence to limiting points
in the interior of $\Theta$, the first step consists in proving that
$\PP$-almost-surely, the number of truncations is finite (see e.g. Andrieu \etal
\cite[Theorem 1]{andrieu:moulines:priouret:2005}). Therefore, the term $\left(
  \theta_0 - \theta_n - \gamma_{n+1} \Xi_{n+1} \right) \un_{ \theta_n +
  \gamma_{n+1} \Xi_{n+1} \notin \K_{n+1}}$ is null for any large $n$ on the set
$\{\lim_q \theta_q = \theta_\star \}$ thus showing that it is part of
$\gamma_{n+1} r_{n+1}^{(1)}$ in the expansion (\ref{eq:definition:theta}).

The condition C\ref{hyp:C1} considers a limiting target $\theta_\star$ which is
assumed to be stable and such that the linear term in the Taylor's expansion of
$h$ at $\theta_\star$ does not vanish (see condition
C\ref{hyp:C1}\ref{hyp:C1:c}). Results for the case of vanishing linear term can
be found in Chen \cite[Section 3.2]{chen:2002}. When $h$ is a gradient function so
that the SA algorithm is a stochastic gradient procedure, the condition
C\ref{hyp:C1}\ref{hyp:C1:a} assumes that $\theta_\star$ is a root of the
gradient. Therefore, our assumptions do not cover the case of constrained
optimization problem with solutions on the boundaries of the constraint set
$\Theta$.  For rates of convergence for these constrained SA algorithms, see
e.g.  Buche and Kushner~\cite{buche:kushner:2001}.

The conditions C\ref{hyp:C3} and C\ref{hyp:C4} are designed to address the case
of multiple targets, a framework which improves on many published results. It
is usually assumed in the literature that there is an unique limiting target
(see e.g. Fabian~\cite{fabian:1978}, Buche and
Kushner~\cite{buche:kushner:2001}, Chen~\cite[Chapter 3]{chen:2002} and Lelong~\cite{lelong:2013}).  While we are interested in proving a Central Limit
Theorem given the tail event $\{\lim_q \theta_q = \theta_\star \}$, it is
assumed in C\ref{hyp:C3}\ref{hyp:C3:a} that the r.v. $e_{n+1}$ in the expansion
(\ref{eq:definition:theta}) is a martingale increment with respect to (w.r.t.)
the probability $\PP$. As discussed above, such an expansion is easily
verified. Note that since the event $\{\lim_q \theta_q = \theta_\star \}$ is in
the tail $\sigma$-field $\sigma(\bigvee_n \F_n)$, it is not true that $\{e_{n},
n \geq 1 \}$ are martingale-increments w.r.t. the probability $\PP(\cdot \vert
\lim_q \theta_q = \theta_\star)$. Therefore, our framework is not a special
case of the single target framework.

The main use of C\ref{hyp:C3} is to prove that the $\{e_n, n\geq 1 \}$
satisfies a CLT under the conditional distribution $\PP(\cdot \vert \lim_q
\theta_q = \theta_\star)$.  We could weaken some of the assumptions, for
example by relaxing the $2+\tau$-moment condition C\ref{hyp:C3}\ref{hyp:C3:b}
which is a way to easily check the Lindeberg condition for martingale
difference array.  Nevertheless, our goal is not only to state a theorem with
weaker assumptions but also to present easy-to-check conditions.

When there exists $\tau>0$ such that $\sup_{k \geq 1} \PE \left[|e_{k}
  |^{2+\tau} \right]<\infty$, C\ref{hyp:C3}\ref{hyp:C3:b} is satisfied with
$\A_m = \A_{m,k} = \Omega$.  When there exist $\tau,\delta >0$ such that
\begin{equation}
  \label{eq:C2b:restrictif}
  \sup_{k \geq 0} \PE \left[|e_{k+1} |^{2+\tau} \un_{|\theta_k
    -\theta_\star| \leq \delta } \right]<\infty \eqsp,
\end{equation}
then C\ref{hyp:C3}\ref{hyp:C3:b} is satisfied with $\A_{m,k} = \bigcap_{m \leq
  j \leq k} \{ |\theta_j -\theta_\star| \leq \delta\}$ and $\A_m = \bigcap_{j
  \geq m} \{ |\theta_j -\theta_\star| \leq \delta\}$.  In most contributions,
rates of convergence are derived under the condition (\ref{eq:C2b:restrictif})
(see e.g. the recent works by Pelletier~\cite{pelletier:1998} and Lelong~\cite{lelong:2013}).
This framework is too restrictive to address the case of SA with controlled
Markov chain dynamics when the ergodic properties of the transition kernels
$\{Q_\theta, \theta \in \Theta \}$ are not uniform in $\theta$. Our assumption
C\ref{hyp:C3}\ref{hyp:C3:b} is designed to address this framework as it will be
shown in Section~\ref{sec:appli}.

C\ref{hyp:C3}\ref{hyp:C3:c} is an assumption on the conditional variance of the
martingale-increment term $e_{n}$, which is more general than what is usually
assumed. In Zhu~\cite{zhu:1996}, Pelletier~\cite{pelletier:1998}, Chen~\cite{chen:2002} and
Leling~\cite{lelong:2013} (resp. in Delyon~\cite{delyon:2000}), a CLT is proved
under the assumption that $\PE \left[e_{k+1}e_{k+1}^{T}|\F_{k} \right] =
U_\star + D_{k}^{(1)} $ (resp. $\PE \left[e_{k+1}e_{k+1}^{T}|\F_{k} \right] =
U_\star + D_{k}^{(2)} $) where $D_k^{(1)}, D_k^{(2)}$ satisfy (\ref{eq:var3})
and $U_\star$ is a deterministic symmetric positive definite matrix. A first
improvement is to remove the assumption that $U_\star$ is deterministic.  A
second improvement is in the combination $D_k^{(1)} + D_k^{(2)}$. The
introduction of the term $D_k^{(2)}$ is a strong improvement since it covers
the case of SA with controlled Markov chain dynamic: observe indeed that in
this case $\PE \left[e_{k+1}e_{k+1}^{T}|\F_{k} \right]$ is a function of $(X_k,
\theta_k)$ and it is really unlikely that this term converges almost-surely to
a (random) variable along the set $\{\lim_q \theta_q = \theta_\star \}$.
Allowing an additional term $D_k^{(2)}$ such that the sum $\sum_{k=1}^n
D_k^{(2)}$ converges in some sense to zero introduces more flexibility (see
Section~\ref{sec:appli} for more details). We will also show in
Section~\ref{sec:appli} how our framework improves on
Delyon~\cite{delyon:2000}.  Examples of SA algorithm where
C\ref{hyp:C3}\ref{hyp:C3:c} holds with resp.  Robbins-Monro and controlled
Markov chain dynamics can be found resp. in Bianchi
\etal~\cite{bianchi:fort:hachem:2013} and Fort
\etal~\cite{fort:jourdain:kuhn:lelievre:stoltz:2013}.

Examples of sequences satisfying the condition C\ref{hyp:C2} are the polynomial
ones. The step size $\gamma_{n}\sim \gamma_\star n^{-a}$ for $a \in (1/2, 1)$
satisfies C\ref{hyp:C2}\ref{hyp:C2:a:1}.  The step size
$\gamma_{n}\sim\gamma_{\star}/n$ satisfies C\ref{hyp:C2}\ref{hyp:C2:a:2}; note
that the condition on ($\gamma_\star, L)$ is well known in the literature (see
e.g. Chen~\cite[Assumption A3.1.4]{chen:2002}).

\subsection{Main result}
\begin{theo}
\label{theo:clt} 
Choose $\theta_0 \in \Theta$ and consider the sequence $\{\theta_n, n \geq 0
\}$ given by (\ref{eq:definition:theta}). Assume C\ref{hyp:C1}, C\ref{hyp:C3},
C\ref{hyp:C4} and C\ref{hyp:C2}. Let $V$ be the positive definite matrix
satisfying w.p.1 on the set $\{\lim_q \theta_q = \theta_\star \}$,
\[
\left\{\begin{array}{ll} V\nabla h(\theta_{\star})^T+\nabla
    h(\theta_{\star})V=-U_{\star}\eqsp, & \text{in case C\ref{hyp:C2}\ref{hyp:C2:a:1}\eqsp,} \\
    V(\Id+2\gamma_{\star}\nabla h(\theta_{\star})^T)+(\Id +2\gamma_{\star}\nabla
    h(\theta_{\star}))V=-2\gamma_{\star}U_{\star}\eqsp, & \text{in case
      C\ref{hyp:C2}\ref{hyp:C2:a:2} \eqsp.}
\end{array}\right.
\]
Under the conditional probability $\PP\left( \cdot \vert \lim_q \theta_q =
  \theta_\star \right)$, $\{\gamma_n^{-1/2} \ \left( \theta_n -
  \theta_\star \right), n\geq 1 \}$ converges in distribution to a r.v. with
characteristic function given for any $t \in \Rset^d$ by
\[
\frac{1}{\PP(\lim_q \theta_q = \theta_\star)} \PE\left[ \un_{\lim_q \theta_q =
    \theta_\star} \ \exp(-\frac{1}{2}t^T V t)\right] \eqsp.
\]
\end{theo}
When the matrix $U_\star$ in Assumption C\ref{hyp:C3}\ref{hyp:C3:c} is
deterministic, the limiting distribution is a centered multidimensional
Gaussian distribution with (deterministic) covariance matrix $V$.

Given matrices $A,E$, existence of a solution to the equation $\quad VA+
A^TV=-E \quad$ is solved by the Lyapunov theorem (see e.g. Horn and
Johnson~\cite[Theorem 2.2.1.]{horn:johnson:1994}). When $A$ is a (negative)
stable real matrix and $E$ is positive definite, then there exists an unique
positive definite matrix $V$ satisfying the Lyapunov equation $\quad VA+
A^TV=-E \quad$ (see e.g.  Horn and Johnson~\cite[Theorem
2.2.3.]{horn:johnson:1994}).

\paragraph{Sketch of the proof of Theorem~\ref{theo:clt}}
The proof of Theorem~\ref{theo:clt} is detailed in Section~\ref{sec:proofs}.
The key ingredient is the Central Limit Theorem for martingale arrays.

As commented in Section~\ref{sec:SAtheory:comment}, $e_{n}$ is not a
martingale-increment w.r.t. the conditional probability $\PP(\cdot \vert \lim_q
\theta_q = \theta_\star)$. To overcome this technical difficulty, we use that
\begin{equation}
  \label{eq:proof:steps}
 e_{n+1} = e_{n+1} \un_{A_n}  + e_{n+1} \left(1 - \un_{A_n} \right) 
\end{equation}
where $\{A_n, n\geq 1\}$ is a $\F_n$-adapted sequence of sets converging to $\{\lim_q
\theta_q = \theta_\star \}$ (such a sequence always exists, see
Lemma~\ref{lem:PierreP}).  Along the event $\{\lim_q \theta_q = \theta_\star
\}$, the second term in the right hand side (rhs) of (\ref{eq:proof:steps}) is
null for any $n$ larger than some almost-surely finite random time.

We  write $ \theta_n - \theta_\star = \mu_n + \rho_n$, where $\mu_n$
satisfies the equation
\[
\mu_{n+1} = \left( \Id + \gamma_{n+1} \nabla h(\theta_\star) \right) \mu_n +
\gamma_{n+1} e_{n+1} \eqsp; \qquad \mu_0 = 0 \eqsp.
\]
$\Id$ denotes the $d \times d$ identity matrix.  Roughly speaking, the sequence
$\{\mu_n, n\geq 0 \}$ captures the linear approximation of $h(\theta_n)$ and
the martingale-increment noise sequence $\{e_n, n\geq 1 \}$.

We prove that $\gamma_n^{-1/2} \ \rho_n \un_{\lim_q \theta_q =
  \theta_\star}$ converges to zero in probability so that $\{\mu_n,n \geq 0 \}$
is the leading term. We then establish that for any $t\in
\Rset^d$,
\[
\lim_n \PE\left[\un_{\lim_q \theta_q = \theta_\star} \ \exp\left( i
    \gamma_n^{-1/2} \ t^T \mu_n\right) \right] = \PE\left[ \un_{\lim_q
    \theta_q = \theta_\star} \ \exp\left( -\frac{1}{2}t^T V t \right) \right]
\eqsp.
\]

\section{A Central Limit Theorem for Iterate Averaging}
\label{sec:averSAtheory}
Theorem~\ref{theo:clt} shows that the rate of convergence of the sequence
$\{\theta_n, n\geq 0 \}$ to $\theta_\star$ is $O(n^{a/2})$ when $\gamma_n \sim
\gamma_\star /n^a$ for some $a \in (1/2, 1]$.  The maximal rate is reached by
choosing $\gamma_n \sim \gamma_\star /n$, for some $\gamma_\star$ satisfying
the conditions C\ref{hyp:C2}\ref{hyp:C2:a:2}.  The main drawback with such a
choice of the step-size sequence $\{\gamma_n, n\geq 1 \}$ is that in practice,
$-L$ \ie the largest real part of the eigenvalues of $\nabla h(\theta_\star)$
is unknown so that the condition C\ref{hyp:C2}\ref{hyp:C2:a:2} is difficult to
check.

The second comment is on the limiting covariance matrix when the rate is
maximal (\ie in the case $\gamma_n \sim \gamma_{\star}/n$). For any
non-singular matrix $\Gamma$, we could define the algorithm
\[
\tilde \theta_{n+1}= \tilde\theta_{n}+\gamma_{n+1} \Gamma
h(\tilde\theta_{n})+\gamma_{n+1} \Gamma e_{n+1}+\gamma_{n+1} \Gamma r_{n+1}
\eqsp, \qquad \tilde{\theta}_0 \in \Theta \eqsp.
\]
This equation is of the form (\ref{eq:definition:theta}) with a mean field
$\tilde h = \Gamma h$ and noises $\{e_n, r_n, n\geq 1 \}$ replaced with $\{
\Gamma e_n, \Gamma r_n, n\geq 1 \}$.  Then, Theorem~\ref{theo:clt} gives
sufficient conditions so that a CLT for the sequence $\{\tilde \theta_{n},
n\geq 0 \}$ holds: the matrix $V$ is replaced with $\tilde V =\tilde V(\Gamma)
$ satisfying
\[
\tilde V(\Id+2\gamma_{\star}\nabla h(\theta_{\star})^T \Gamma^T)+(\Id
+2\gamma_{\star}\nabla h(\theta_{\star})\Gamma)\tilde V=-2\gamma_{\star}\Gamma
U_{\star}\Gamma^T \eqsp.
\]
A natural question is the ``optimal'' choice of the gain matrix $\Gamma$,
defined as the matrix $\Gamma_\star$ such that for any $\lambda \in \Rset^d$,
$\lambda^T \tilde V(\Gamma) \lambda \geq \lambda^T \tilde V(\Gamma_\star)
\lambda$.  Following the same lines as in Benveniste \etal~\cite[Proposition 4,
Chapter 3, Part I]{benveniste:metivier:priouret:1987}, it can be proved that
$\Gamma_\star = - \gamma_\star^{-1} \nabla h(\theta_\star)^{-1}$ and in this
case,
\[
\tilde V(\Gamma_\star) = \gamma_\star^{-1} \nabla h(\theta_\star)^{-1}
U_\star \nabla h(\theta_\star)^{-T} \eqsp.
\]
Theorem~\ref{theo:clt:aver} below shows that by considering the averaged
sequence $\{\btheta_n, n\geq 0\}$, the optimal rate of convergence (\ie the
rate $\sqrt{n}$) and the optimal asymptotic covariance matrix (optimal in the
sense discussed above) can be reached whatever the sequence $\{\gamma_n, n\geq
1\}$ satisfying C\ref{hyp:C2}\ref{hyp:C2:a:1} used in the basic SA sequence
(\ref{eq:definition:theta}).  Therefore, such an optimality can be obtained
even when $\nabla h(\theta_\star)$ is unknown. Note also that on a practical
point of view, slow decreasing step-size $\gamma_n$ are better (see e.g.
Spall~\cite[Section 4.4.]{spall:2003}) and this simple averaging procedure
improves the rate of convergence of the estimate of $\theta_\star$.

These properties of the averaged sequence were simultaneously established by
Ruppert~\cite{ruppert:1991} and Polyak and Juditsky~\cite{polyak:juditsky:1992}
under more restrictive conditions than those stated below.

\subsection{Assumptions}
\begin{hypAver}\label{hyp:Aver3}
 \begin{enumerate}[(a)]
 \item \label{hyp:Aver3:a}  $\{e_{n},\ n\geq 1\}$ is a $\F_n$-adapted
   $\PP$-martingale-increment sequence.
 \item \label{hyp:Aver3:b} There exists a sequence $\{\A_m, m\geq 1 \}$ such
   that $\lim_m \PP(\A_m \vert \lim_q \theta_q = \theta_\star) =1$, and for any
   $m \geq 1$,
    \[\sup_{k} \PE \left[|e_{k}
      |^{2}\un_{\A_{m,k-1}} \right]<\infty \eqsp,
\]
where $\A_{m,k-1} \in \F_{k-1}$ and $\lim_k \un_{\A_{m,k}} = \un_{\A_m}$ almost-surely on the set $\{\lim_q \theta_q = \theta_\star\}$.
  \item \label{hyp:Aver3:c}  Let
\[
\mathcal{E}_{n+1} = \frac{1}{\sqrt{n+1}}\sum_{k=0}^{n}e_{k+1} \eqsp.
\]
There exists a random matrix $U_\star$, positive definite w.p.1. on the set
$\{\lim_q \theta_q = \theta_\star \}$, such that for any $t \in \Rset^d$,
\[
\lim_n \PE\left[\un_{\lim_q \theta_q = \theta_\star} \ \exp(it^T
  \mathcal{E}_{n+1}) \right] = \PE\left[\un_{\lim_q \theta_q = \theta_\star} \ 
  \exp(- \frac{1}{2} \ t^T U_\star t) \right] \eqsp.
\]
\end{enumerate}
\end{hypAver}
We prove in Lemma~\ref{lem:AVER1} that when $\lim_n n \gamma_n >0$, assumption
C\ref{hyp:C3} implies AVER\ref{hyp:Aver3}.  Note also that since $\lim_m
\PP(\A_m \vert \lim_q \theta_q = \theta_\star) =1$,
AVER\ref{hyp:Aver3}\ref{hyp:Aver3:c} is equivalent to the condition: for any
$m \geq 1$,
\[
\lim_n \PE\left[\un_{\lim_q \theta_q = \theta_\star} \ \exp(it^T
  \mathcal{E}_{n+1}) \un_{\A_m} \right] = \PE\left[\un_{\lim_q \theta_q = \theta_\star} \ 
  \exp(- \frac{1}{2} \ t^T U_\star t)  \un_{\A_m} \right] \eqsp.
\]
For a sequence of $\Rset^d$-valued r.v. $\{Z_n, n \geq 0\}$, we write $Z_n =
O_{L^p}(1)$ if $\sup_n \PE[|Z_n|^p] < \infty$.

\begin{hypAver}\label{hyp:Aver4} $r_n$ is $\F_n$-adapted. $r_n = r_n^{(1)} + r_n^{(2)}$ with  for any $m \geq 1$,
 \begin{enumerate}[(a)] 
 \item \label{hyp:Aver4:a} $ \gamma_n^{-1/2} \ r_{n}^{(1)}
   \un_{\lim_{q}\theta_{q}=\theta_{\star}} \un_{\A_m}= \ O_{w.p.1}(1)
   O_{L^2}(1)$.
 \item \label{hyp:Aver4:a'} $\sqrt{\gamma_n} \sum_{k=1}^n r_k^{(2)} \un_{\lim_q
     \theta_q = \theta_\star} \un_{\A_m} = O_{w.p.1}(1) O_{L^2}(1)$ .
 \item
   \label{hyp:Aver4:b} $n^{-1/2} \sum_{k=0}^n r_{k+1}   \un_{\lim_{q}\theta_{q}=\theta_{\star}} \plim 0$.
\end{enumerate}
The sequence $\{\A_m, m \geq 1 \}$ is defined in
AVER\ref{hyp:Aver3}\ref{hyp:Aver3:b}.
\end{hypAver}
Note that AVER\ref{hyp:Aver4}\ref{hyp:Aver4:b} is equivalent to $n^{-1/2}
\sum_{k=0}^n r_{k+1} \un_{\lim_{q}\theta_{q}=\theta_{\star}} \un_{\A_m} \plim
0$ for any $m\geq 1$.
 
\begin{hypAver}\label{hyp:Aver2}
  $\lim_{n}n\gamma_{n}=+\infty$ and
\[
\lim_{n}\frac{1}{\sqrt{n}}\sum_{k=1}^{n}\gamma_{k}^{-1/2}
\left|1-\frac{\gamma_{k}}{\gamma_{k+1}} \right|=0 \eqsp, \qquad \qquad
\lim_{n}\frac{1}{\sqrt{n}}\sum_{k=1}^{n} \gamma_k =0.
\]
\end{hypAver}
The step size $\gamma_{n}\sim \gamma_\star n^{-a}$ for $a \in (1/2, 1)$
satisfies AVER\ref{hyp:Aver2} but the step size
$\gamma_{n}\sim\gamma_{\star}/n$ does not.  Observe that if the sequence
$\{\gamma_n, n\geq 0 \}$ is non-increasing (or ultimately non-increasing) then
(see Lemma~\ref{lem:AVER3})
\[
\lim_n n \gamma_n =+\infty \Longrightarrow \lim_n
\frac{1}{\sqrt{n}}\sum_{k=1}^{n}\gamma_{k}^{-1/2}
\left|1-\frac{\gamma_{k}}{\gamma_{k+1}} \right|=0 \eqsp.
\]

\subsection{Main results}
We show that the above conditions allow a control of the $L^2$-moment of the
errors $\{\theta_n - \theta_\star, n\geq 0 \}$. This result is a cornerstone
for the proof of Theorem~\ref{theo:clt:aver}. The proof is given in
Section~\ref{sec:proofs}.
\begin{prop}
\label{prop:L2norm:theta}
Assume C\ref{hyp:C1}, C\ref{hyp:C2},
AVER\ref{hyp:Aver3}\ref{hyp:Aver3:a}-\ref{hyp:Aver3:b} and
AVER\ref{hyp:Aver4}\ref{hyp:Aver4:a}-\ref{hyp:Aver4:a'}. Then, for any $m \geq 1$
\[
\gamma_n^{-1} \left\| \theta_n - \theta_\star \right\|^2 \un_{\lim_q
  \theta_q = \theta_\star}  \un_{\A_m} = O_{w.p.1}(1) \ O_{L^1}(1) \eqsp.
\]
\end{prop}

\begin{theo}
\label{theo:clt:aver}
Choose $\theta_0 \in\Theta$ and consider the averaged sequence given
by~(\ref{eq:averaged:sequence:def}).  Assume C\ref{hyp:C1},
C\ref{hyp:C2}\ref{hyp:C2:a:1}, AVER\ref{hyp:Aver3}, AVER\ref{hyp:Aver4} and
AVER\ref{hyp:Aver2}.  Then for any $t \in \Rset^d$,
  \begin{multline*}
    \lim_n \PE\left[\un_{\lim_q \theta_q = \theta_\star} \ \exp\left( i
        \sqrt{n} \ 
        t^T \left( \btheta_n - \theta_\star \right)\right) \right]  \\
    = \PE\left[\un_{\lim_q \theta_q = \theta_\star} \ \exp \left(-
        \frac{1}{2}t^T \nabla h(\theta_\star)^{-1} \ U_\star \ (\nabla
        h(\theta_\star)^{-1})^T t \right) \right] \eqsp.
  \end{multline*}
\end{theo}
\paragraph{Sketch of the proof of Theorem~\ref{theo:clt:aver}}
The proof is detailed in Section~\ref{sec:proofs}. Since $\lim_m \PP(\A_m \vert
\lim_q \theta_q = \theta_\star) =1$, we only have to prove that for any $m \geq
1$ and $t \in \Rset^d$,
 \begin{multline*}
    \lim_n \PE\left[\un_{\lim_q \theta_q = \theta_\star} \un_{\A_m} \ \exp\left( i
        \sqrt{n} \ 
        t^T \left( \btheta_n - \theta_\star \right)\right) \right]  \\
    = \PE\left[\un_{\lim_q \theta_q = \theta_\star}  \un_{\A_m} \ \exp \left(-
        \frac{1}{2}t^T \nabla h(\theta_\star)^{-1} \ U_\star \ (\nabla
        h(\theta_\star)^{-1})^T t \right) \right] \eqsp.
  \end{multline*}
 We write
\[
\btheta_n - \theta_\star = - \frac{\nabla h(\theta_\star)^{-1}}{n+1}
\sum_{k=0}^n e_{k+1} + Z_n \eqsp.
\]
We show that $\sqrt{n} Z_n \un_{\lim_q \theta_q = \theta_\star} \un_{\A_m}$
converges to zero in probability for any $m \geq 1$; for this step, the main
tool is Proposition~\ref{prop:L2norm:theta}.  The proof is then concluded by
AVER\ref{hyp:Aver3}\ref{hyp:Aver3:c}.

\section{Application to  SA  with controlled Markov chain dynamics} 
\label{sec:appli}
Let $\{\K_n,n\geq 0 \}$ be a sequence of compact subsets of $\Theta$ such that
\[
\K_n \subseteq \K_{n+1} \eqsp, \qquad \qquad \bigcup_{n \geq 0} \K_n = \Theta \eqsp.
\]
Let $\{Q_\theta, \theta \in \Theta \}$ be a family of Markov transition kernels
onto $(\mathbb{X}, \mathcal{X})$.  We consider the following SA algorithm with
truncation at randomly varying bounds: $\theta_0 \in \K_0, \sigma_0 = 0$ and for
$n \geq 0$,
\begin{list}{}{}
\item set $\theta_{n+1/2} = \theta_n + \gamma_{n+1} H(\theta_n X_{n+1})$.
\item update
\[
(\theta_{n+1}, \sigma_{n+1}) = \left\{ 
  \begin{array}{ll}
(\theta_{n+1/2}, \sigma_n) \eqsp, & \ \text{if $\theta_{n+1/2} \in \K_{\sigma_{n}}$} \eqsp, \\
(\theta_0 , \sigma_n + 1)& \ \text{otherwise,}
  \end{array}
\right. 
\]\end{list}
where $\{X_n,n \geq 0 \}$ is a controlled Markov chain on $(\Omega,
\mathcal{A}, \PP)$ with conditional distribution given by
\begin{equation}
  \label{eq:appli:LawX}
  \PP(X_{n+1} \in A \vert \F_n) = Q_{\theta_n}(X_n, A) \eqsp, \qquad \F_n =
\sigma(\theta_0, X_0, \cdots, X_n) \eqsp.
\end{equation}
The random sequence $\{\sigma_n, n \geq 0 \}$ is a non-negative integer-valued
sequence counting the number of truncations.  Such a truncated SA was
introduced by Chen \etal~\cite{chen:guo:gao:1988} (see also Chen~\cite[Chapter
2]{chen:2002}) to address the boundedness problem of the SA sequence
$\{\theta_n, n \geq 0\}$.  A more general truncated SA algorithm with
controlled Markov chain dynamics is introduced in Andrieu
\etal~\cite{andrieu:moulines:priouret:2005}: when truncation occurs, both the
parameter $\theta_{n+1/2}$ and the draw $X_n$ used to obtain the next point
$X_{n+1}$ are modified.

The key point of the proof of convergence of this algorithm is to show that the
number of truncations is finite with probability one, so that after some random
time, the sequence $\{\theta_n, n\geq 0 \}$ is almost-surely bounded and obeys
the update rule $\theta_{n+1} = \theta_n + \gamma_{n+1} H(\theta_n, X_{n+1})$.
Conditions implying almost-sure boundedness and almost-sure convergence of the
sequence $\{\theta_n, n \geq 0 \}$ when $\{X_n,n \geq 0 \}$ is a controlled
Markov chain can be found in Andrieu \etal~\cite[Section
3]{andrieu:moulines:priouret:2005}.  Since in this paper we are interested in
CLT's, we will assume that
\begin{hypA} \label{hyp:appli:limit} 
  \begin{enumerate}[(a)]
\item \label{hyp:appli:ergo1} For any $\theta \in \Theta$, there exists a probability distribution
     $\pi_\theta$ on $(\mathbb{X}, \mathcal{X})$ such that $\pi_\theta Q_\theta
     = \pi_\theta$. Set 
     \begin{equation}
       \label{eq:appli:definition:h}
       h(\theta) = \int H(\theta,x) \ \pi_\theta(dx) \eqsp.
     \end{equation}
  \item \label{hyp:appli:limit:item} the number of truncations is finite with probability one: $\PP(\limsup_n
  \sigma_n < \infty) =1$ and there exists $\theta_\star \in \Theta$ satisfying
  C\ref{hyp:C1} such that $\PP(\lim_n \theta_n = \theta_\star) >0$.
  \end{enumerate}
\end{hypA}
 For simplicity, we consider the case when $H$ is bounded and the step-size is
 polynomially decreasing. Extensions to the case $H$ is unbounded can be done
 along the same lines as in Andrieu \etal~\cite{andrieu:moulines:priouret:2005}.
\begin{hypA}\label{hyp:appli:HandGamma}
  \begin{enumerate}[(a)]
  \item for any compact set $\K \subseteq \Theta$, $\sup_{\theta \in \K}
    \sup_{x \in \mathbb{X}} |H(\theta,x)| < \infty$.
  \item There exists $a \in (1/2, 1]$ such that $\gamma_n = \gamma_\star /
    n^a$. When $a=1$, $\gamma_\star$ satisfies the condition
    C\ref{hyp:C2}\ref{hyp:C2:a:2}.
  \end{enumerate}
\end{hypA}
 We assume that the transition kernels $\{Q_\theta, \theta \in \Theta \}$ satisfy
 \begin{hypA} \label{hyp:appli:ergo}
   \begin{enumerate}[(a)]
   \item \label{hyp:appli:ergo2} For any $\theta \in \Theta$, there exists a measurable function
     $\hH_\theta: (\mathbb{X}, \mathcal{X}) \to (\Rset^d, \mathcal{B}(\Rset^d))$ such that
     \begin{equation}
       \label{eq:appli:PoissonH}
       H(\theta,x) - h(\theta) = \hH_\theta(x) - Q_\theta \hH_\theta(x) \eqsp.
     \end{equation}
There exists a function $V_1: \mathcal{X} \to [1, \infty)$ such that for any
compact subset $\K \subseteq \Theta$, 
\begin{equation}
  \label{eq:appli:boundhH}
  \sup_{\theta \in \K} \sup_{x \in \mathbb{X}} \frac{|\hH_\theta(x)| + | Q_\theta
  \hH_\theta(x)|}{V_1(x)} < \infty \eqsp.
\end{equation}
\item \label{hyp:appli:ergo3} For any $\theta \in \Theta$, there exists a measurable function
     $U_\theta: (\mathbb{X}, \mathcal{X}) \to (\Rset^{d^2}, \mathcal{B}(\Rset^{d^2}))$ such that
     \begin{equation}
       \label{eq:appli:PoissonF}
       F_\theta(x) - \int F_\theta(x) \ \pi_\theta(dx) = U_\theta(x) - Q_\theta
U_\theta(x) \eqsp, 
     \end{equation}
where $F_\theta(x) = \int Q_\theta(x,dy) \,  \hH_\theta(y)\hH_\theta(y)^T - Q_\theta \hH_\theta(x) \ 
\left( Q_\theta \hH_\theta(x) \right)^T$.  There exists a function $V_2:
\mathcal{X} \to [1, \infty)$ such that for any compact subset $\K \subseteq
\Theta$, 
\begin{equation}
  \label{eq:appli:boundU}
\sup_{\theta \in \K} \sup_{x \in \mathbb{X}} \frac{|U_\theta(x)|  + |Q_\theta U_\theta(x)|}{V_2(x)} <
\infty \eqsp.
\end{equation}
\item \label{hyp:appli:ergo4} There exist $\delta, \tau, \bar{\tau} >0$ such
  that for any $m \geq 1$,
  \begin{align*}
    &\sup_{k \geq m} \PE\left[\left( V_1^{2+\tau}(X_{k+1}) + V_2^{1+\bar
          \tau}(X_{k+1}) \right) \un_{ \bigcap_{m \leq j \leq k} \{|\theta_j -
        \theta_\star| \leq
        \delta\}} \right] < \infty \eqsp, \\
  &  \PE\left[ V_1^{2+\tau}(X_{m}) + V_2^{1+\bar \tau}(X_{m})
    \right] < \infty \eqsp.
  \end{align*}
 \item \label{hyp:appli:ergo5} 
For any compact subset $\K \subseteq \Theta$, there exist $b,C>0$ such that for any $\theta, \theta' \in \K$,
\begin{align*}
 &  \left| Q_\theta \hH_\theta(x) - Q_{\theta'} \hH_{\theta'}(x) \right| \leq C \, |\theta - \theta'|^{1/2+b} \,  V_1(x)  \eqsp, \\
 &  \left| U_\theta(x) - U_{\theta'}(x) \right| \leq C \, |\theta - \theta'|^b \,
V_2(x) \eqsp.
\end{align*}
Furthermore, almost-surely
\[
\lim_n \left( \int F_{\theta_n}(x) \ \pi_{\theta_n}(dx) - \int
  F_{\theta_\star}(x) \ \pi_{\theta_\star}(dx)\right) \un_{\lim_q \theta_q =
  \theta_\star} = 0 \eqsp.
\]   \end{enumerate}
 \end{hypA}
 Conditions implying the existence of $\pi_\theta$ and solutions to the Poisson
 equations (\ref{eq:appli:PoissonH}) and (\ref{eq:appli:PoissonF}) can be found
 e.g. in Hernandez-Lerma and Lasserre~\cite[Chapter 8]{hernandez:lasserre:2003}
 or in Meyn and Tweedie~\cite[Chapter 17]{meyn:tweedie:2009}.  When the
 transition kernel $Q_\theta$ is uniformly ergodic, then $V_1=V_2$ and is equal
 to the constant function $1$. When the kernel is $V$-geometrically ergodic, we
 can choose $V_1=V^{1/p}, V_2 = V^{2/p}$ for any $p \geq 2$. Sufficient
 conditions for \eqref{eq:appli:boundhH} and \eqref{eq:appli:boundU} based on
 Lyapunov drift inequalities when the chain is geometrically ergodic (resp.
 subgeometrically ergodic) are given by Fort \etal~\cite[Lemma
 2.3]{fort:moulines:priouret:2012} (resp.  Andrieu
 \etal~\cite{andrieu:fort:vihola:2013}.  Andrieu \etal~\cite[Proposition
 6.1.]{andrieu:moulines:priouret:2005} gives sufficient conditions to check
 A\ref{hyp:appli:ergo}\ref{hyp:appli:ergo4} (compare this assumption with the
 condition A3(ii) of Andrieu \etal) when the kernels are $V$-geometrically
 ergodic: in this case, for any $p \geq 2$ we can choose $V_1=V^{1/p}, V_2 =
 V^{2/p}$ and $2(1+\bar \tau)/p =1$. The first set of conditions in
 A\ref{hyp:appli:ergo}\ref{hyp:appli:ergo5} is an assumption on the
 regularity-in-$\theta$ of the solution to the Poisson equation.  Andrieu
 \etal~\cite[Proposition 6.1.]{andrieu:moulines:priouret:2005} give sufficient
 conditions in terms of the regularity-in-$\theta$ of the transition kernels
 $Q_\theta$. When $\pi_\theta = \pi$ for any $\theta$, the second set of
 conditions can be established by combining smoothness-in-$\theta$ properties
 of the function $F_\theta$ and the dominated convergence theorem. When
 $\pi_\theta$ depends on $\theta$, Fort \etal~\cite[Theorem 2.11 and
 Proposition 4.3]{fort:moulines:priouret:2012} give sufficient conditions for
 this condition to hold.

 We now show how these assumptions imply the conditions C\ref{hyp:C1} to
 C\ref{hyp:C2}. Under A\ref{hyp:appli:limit}\ref{hyp:appli:limit:item}, the condition C\ref{hyp:C1}
 holds; note also that the conditional probability $\PP(\cdot \vert \lim_q
 \theta_q = \theta_\star)$ is well defined. By using
 (\ref{eq:appli:definition:h}) and (\ref{eq:appli:PoissonH}), we write the
 truncated SA algorithm on the form (\ref{eq:definition:theta}) by setting
 \begin{align*}
   e_{n+1} & = \hH_{\theta_n}(X_{n+1}) - Q_{\theta_n} \hH_{\theta_n}(X_{n}) \eqsp, \\
   r_{n+1} & = Q_{\theta_n}\hH_{\theta_n}(X_n) - Q_{\theta_n}
   \hH_{\theta_n}(X_{n+1}) + (\theta_0 - \theta_{n+1/2}) \un_{\theta_{n+1/2} \notin \K_{\sigma_n}} \eqsp.
 \end{align*}
 Let us prove that the condition C\ref{hyp:C3} holds. Since $\theta_n \in
 \F_n$, Eq.~(\ref{eq:appli:LawX}) implies C\ref{hyp:C3}\ref{hyp:C3:a}. Fix
 $\delta$ such that $B(\theta_\star, \delta) = \{\theta \in \Rset^d, | \theta -
 \theta_\star | \leq \delta \} \subseteq \Theta$. Set 
\[
\A_{m,k} = \left\{
  \begin{array}{ll}
\emptyset & \ \text{if $k<m$}, \\
\bigcap_{m \leq j \leq k} \{ |\theta_j - \theta_\star| \leq \delta,  \theta_j = \theta_{j-1/2} \} & \ \text{otherwise.}
  \end{array}
\right.
\]
Then for any $k,m,$ $\A_{m,k} \in \F_k$; $\lim_k \A_{m,k} = \A_m$ where $\A_m =
\bigcap_{j \geq m} \{ | \theta_j - \theta_\star| \leq \delta, \theta_j =
\theta_{j -1/2} \}$; and $\lim_m \PP(\A_m \vert \lim_q \theta_q = \theta_\star)
=1$ by A\ref{hyp:appli:limit}\ref{hyp:appli:limit:item}. Fix $m \geq 1$; by (\ref{eq:appli:boundhH})
applied with $\K = B(\theta_\star, \delta)$, there exists a constant $C$ such
that
\[
\PE\left[ |e_{k+1}|^{2+\tau} \un_{\A_{m,k}} \right] \leq C \, \PE\left[ \left(
    V_1^{2+\tau}(X_k) + V_1^{2+\tau}(X_{k+1}) \right) \un_{\A_{m,k}} \right] \eqsp,
\]
A\ref{hyp:appli:ergo}\ref{hyp:appli:ergo4} concludes the proof of
C\ref{hyp:C3}\ref{hyp:C3:b}. Observe that $ \PE\left[e_{k+1} e_{k+1}^T \vert
  \F_k \right] = F_{\theta_k}(X_k)$.  By using (\ref{eq:appli:PoissonF}), we
write $\PE\left[ e_{k+1} e_{k+1}^T \vert \F_k \right] = U_\star + D_k^{(1)} +
D_k^{(2,a)} +D_k^{(2,b)}$ with
\begin{align*}
  U_\star & = \int F_{\theta_\star}(x) \ \pi_{\theta_\star}(dx) \eqsp, \\
  D_k^{(1)} & = \int  F_{\theta_k}(x) \ \pi_{\theta_k}(dx) - \int F_{\theta_\star}(x) \ \pi_{\theta_\star}(dx) \eqsp, \\
  D_k^{(2,a)} & =  U_{\theta_k}(X_{k+1}) - Q_{\theta_k}
    U_{\theta_k}(X_k) \eqsp, \\
D_k^{(2,b)} & =  U_{\theta_k}(X_k) - U_{\theta_k}(X_{k+1}) \eqsp.
\end{align*}
By A\ref{hyp:appli:ergo}\ref{hyp:appli:ergo5}, $D_k^{(1)} \aslim 0$ on the set
$\{\lim_q \theta_q = \theta_\star \}$.  By (\ref{eq:appli:LawX}),
$\PE\left[D_k^{(2,a)} \vert \F_{k-1} \right]=0$; by application of the
Burkholder inequality (see e.g. Hall and Heyde~\cite[Theorem
2.10]{hall:heyde:1980}), it holds for any $A_k \in \F_k$ such that $\lim_k A_k
= \{\lim_q \theta_q = \theta_\star \}$
\[
\PE\left[ \left| \sum_{k=1}^n D_{k}^{(2,a)} \un_{A_{k}} \un_{\A_{m,k}}
  \right|^{1+\bar \tau} \right] \leq C \, n^{1 \vee (1+\bar \tau)/2} \eqsp.
\]
The constant $C$ is finite since under
A\ref{hyp:appli:ergo}\ref{hyp:appli:ergo4}, $\sup_k \PE\left[ | D_k^{(2,a)}
  |^{1+ \bar \tau} \un_{\A_{m,k}} \right] < \infty$.  Furthermore,
\[
\sum_{k=m}^n D_k^{(2,b)} = U_{\theta_m}(X_m) - U_{\theta_n}(X_{n+1}) +
\sum_{k=m+1}^n \left( U_{\theta_k}(X_k) - U_{\theta_{k-1}}(X_k)\right)
\]
so that by A\ref{hyp:appli:ergo}\ref{hyp:appli:ergo4}-\ref{hyp:appli:ergo5},
there exists a constant $C$ such that 
\[
\PE\left[ \left| \sum_{k=1}^n D_k^{(2,b)}\right|^{1+\bar \tau} \un_{\A_m}
  \un_{\lim_q \theta_q = \theta_\star} \right] \leq C \left(1 +
  \sum_{k=m}^{n-1} \gamma_{k}^{b(1+\bar \tau)} \right) \eqsp.
\]
The above discussion shows that C\ref{hyp:C3}\ref{hyp:C3:c} is verified if $a
> 1/2 \vee 1/(1+\bar \tau)$.

Finally, let us study $r_n$. We write $r_{n+1} = r_{n+1}^{(1)} + r_{n+1}^{(2)}$ with
\[
r_{n+1}^{(1)} = \left( \theta_0 -\theta_{n+1/2} \right) \un_{\theta_{n+1/2} \notin
  \K_{\sigma_n}} + Q_{\theta_{n+1}} \hH_{\theta_{n+1}}(X_{n+1}) - Q_{\theta_n}
\hH_{\theta_n}(X_{n+1}) \eqsp.
\]
By A\ref{hyp:appli:limit}\ref{hyp:appli:limit:item} and A\ref{hyp:appli:ergo}\ref{hyp:appli:ergo5},
$\gamma_n^{-1/2} r_n^{(1)} \un_{\lim_q \theta_q = \theta_\star} \un_{\A_m}
=o_{w.p.1}(1) + o_{L^1}(1)$ for any fixed $m \geq 1$. In addition, by
(\ref{eq:appli:boundhH}), there exists a constant $C$ such that
\[
\PE\left[ \left| \sum_{k=1}^n r_k^{(2)} \right| \un_{\A_m} \right] \leq
\PE\left[V_1(X_1) \right] +  \PE\left[V(X_{n+1}) \un_{\A_m} \right] \eqsp;
\]
and by A\ref{hyp:appli:ergo}\ref{hyp:appli:ergo4}, this term is uniformly
bounded in $n$.

The above discussion is summarized in the following  proposition
\begin{prop}
  Assume A\ref{hyp:appli:limit}, A\ref{hyp:appli:HandGamma} and
  A\ref{hyp:appli:ergo}. If $a \in (1/2 \vee 1/(1+\bar \tau), 1]$, the conditions
  C\ref{hyp:C1} to C\ref{hyp:C2} are satisfied  and
\[
U_\star = \int \pi_{\theta_\star}(dx) \left( \hH_{\theta_\star}(x) \, \hH_{\theta_\star}(x)^T -
  Q_{\theta_\star} \hH_{\theta_\star}(x) \ \left( Q_{\theta_\star} \hH_{\theta_\star}(x) \right)^T \right) \eqsp.
\]
\end{prop}
By application of Theorem~\ref{theo:clt}, we obtain a CLT for randomly
truncated SA with controlled Markov chain dynamics. Our result improves on
Delyon~\cite[Theorem 25]{delyon:2000}. Under stronger conditions (for example, it is
assumed that $V_1$ and $V_2$ are bounded functions; there is a single target
$\theta_\star$), Delyon~\cite{delyon:2000} establishes a CLT in the case $\gamma_n =
\gamma_\star /n^a$ with the condition $a \in (2/3,1]$. Note that if $V_1,V_2$ are bounded
then A\ref{hyp:appli:ergo}\ref{hyp:appli:ergo4} holds with any $\bar \tau>0$ so
that our approach only requires $a \in (1/2, 1]$ which is the usual range of
values for SA algorithms.

Using similar tools, the conditions of Theorem~\ref{theo:clt:aver} can be
verified; details are left to the interested reader.

\section{Proof}\label{sec:proofs}
\subsection{Definitions and Notations}
\label{sec:notations}
Let $\{A_n,n \geq 0\}$ be a sequence of sets such that
\begin{equation}
  \label{eq:definition:Ansequence}
 A_n \in \F_n \eqsp, \qquad \qquad  \lim_n
\un_{A_n} = \un_{\lim_q \theta_q = \theta_\star} \qquad \text{w.p.1} \eqsp.
\end{equation}
Such a sequence exists by Lemma~\ref{lem:PierreP}. Define recursively two sequences 
\begin{eqnarray}
  \label{eq:definition:mu}
  \mu_{n+1}&=&(\Id+\gamma_{n+1} \nabla h(\theta_{\star}))\mu_{n}+\gamma_{n+1}e_{n+1}  \eqsp, \qquad \qquad \mu_{0}=0 \eqsp; \\
  \label{eq:definition:rho}
  \rho_{n+1}&=&\theta_{n+1}-\theta_{\star}-\mu_{n+1} \eqsp, \qquad \qquad  \rho_{0}=\theta_{0}-\theta_{\star} \eqsp;
\end{eqnarray}
and the matrices $\psi_\star(n,k)$ for $1 \leq k\leq n$,
\begin{equation}
\label{eq:definition:psistar}
 \psi_{\star}(n, k) = \prod_{j=k}^{n}(\Id+\gamma_{j}\nabla h(\theta_{\star})) \eqsp.
\end{equation}
By convention, $\psi_{\star} (n, n +1) = \Id$. Under
C\ref{hyp:C1}\ref{hyp:C1:a}-\ref{hyp:C1:b}, there exist a set of random $d
\times d$ symmetric matrices $\{R_{i}^{(n)},\ i\leq d\}$ such that the entry
$i$ of the column vector $\{h(\theta_{n})-\nabla
h(\theta_{\star})(\theta_{n}-\theta_{\star})\}$ is equal to $\quad
(\theta_{n}-\theta_{\star})^{T}R_{i}^{(n)}(\theta_{n}-\theta_{\star})$ \quad.
More precisely,
\begin{equation}
  \label{eq:definition:matrixR}
  R_{i}^{(n)}(k,\ l) = \int_{0}^{1}\frac{1}{2}(1-t)^{2}\frac{\partial^{2}h_{i}}{\partial\theta_{k}\partial\theta_{l}}(\theta_{n}+t(\theta_{n}-\theta_{\star})) \ dt \eqsp. 
\end{equation}
Let $R_\bullet^{(n)}$ be the tensor such that 
\begin{equation}
  \label{eq:Taylor:h}
  h(\theta_{n}) = \nabla
h(\theta_{\star})(\theta_{n}-\theta_{\star}) +  (\theta_{n}-\theta_{\star})^{T}R_{\bullet}^{(n)}(\theta_{n}-\theta_{\star}) \eqsp.
\end{equation}
Finally, for $1 \leq k\leq n$, define the $d \times d$ matrices
\begin{equation}
  \label{eq:definition:psi}
  \psi(n, k) = \prod_{j=k}^{n}(\Id+\gamma_{j} \{\nabla h(\theta_{\star})+2\mu_{j-1}^{T}R^{(j-1)}_\bullet+\rho_{j-1}^{T}R^{(j-1)}_\bullet\}) \eqsp,
\end{equation}
with the convention that $\psi(n,n+1) = \Id$.

\subsection{Preliminary results on the sequence $\{\mu_n, n\geq 0\}$}
By iterating (\ref{eq:definition:mu}), we have by definition of $\psi_\star$
(see (\ref{eq:definition:psistar}))
\begin{equation}
  \label{eq:expression:mu}
  \mu_{n+1}=\sum_{k=1}^{n+1}\gamma_{k}\psi_{\star}(n+1,k+1)e_{k} \eqsp.
\end{equation}

\begin{prop}
  \label{prop:mu:ps-L2}
  Assume C\ref{hyp:C1}\ref{hyp:C1:b}-\ref{hyp:C1:c},
  C\ref{hyp:C3}\ref{hyp:C3:a}-\ref{hyp:C3:b} and C\ref{hyp:C2}. Then
  \begin{enumerate}[(i)]
  \item $\mu_n \un_{\lim_q \theta_q = \theta_\star} \aslim 0$ when $n \to
    \infty$.
  \item for any $m \geq 1$, $\gamma_{k}^{-1}|\mu_{k}|^{2} \ \un_{\lim_q
      \theta_q = \theta_\star} \un_{\A_m} = O_{L^1}(1) +o_{w.p.1}(1)$.
  \end{enumerate}
\end{prop}
\begin{proof}
 Let $m \geq 1$ be fixed.  Set
$\mu_{n+1} = \mu_{n+1}^{(1)} + \mu_{n+1}^{(2)}$, with
\[
\mu_{n+1}^{(1)} = \sum_{k=1}^{n+1} \gamma_k
\psi_{\star}(n+1,k+1)e_{k} \un_{\A_{m,k-1}} \eqsp.
\]
\textit{(i)} Since
\[
\PP\left( \bigcup_m \A_m \Big \vert \lim_q \theta_q = \theta_\star \right) \geq
\lim_{M} \PP\left( \A_M \big \vert \lim_q \theta_q =\theta_\star\right) =1 \eqsp,
\]
we only have to prove that for any $m \geq 1$, $\lim_n \mu_n \un_{\A_m}
\un_{\lim_q \theta_q = \theta_\star} \aslim 0$.  Let $m\geq 1$.  Let us first
consider $\mu_n^{(1)}$ and define $\quad S_{n} \eqdef \sum_{k\geq
  n}\gamma_{k}e_{k}\un_{\A_{m,k-1}} \quad$.  By
(\ref{eq:expression:mu}) and the Abel transform, we have
\begin{align}
\label{eq:Abelsurmu1}
  \mu_{n+1}^{(1)} & =\sum_{k=1}^{n+1}\psi_{\star}(n+1, k+1)(S_{k}-S_{k+1})  \nonumber \\
  & =
  (S_{n+1}-S_{n+2})+\sum_{k=1}^{n}\psi_{\star}(n+1,k+1)S_{k}-\sum_{k=2}^{n+1}\psi_{\star}(n+1,
  k)S_{k} \nonumber \\
  & =
  -S_{n+2}+\psi_{\star}(n+1,2)S_{1}+\sum_{k=2}^{n}(\psi_{\star}(n+1,k+1)-\psi_{\star}(n+1,
  k))S_{k} \nonumber \\
&= -S_{n+2}+\psi_{\star}(n+1,2)S_{1}-\sum_{k=2}^{n}\gamma_{k}\psi_{\star}(n+1, k+1)\nabla h(\theta_{\star})S_{k}
\end{align}
where we used (\ref{eq:definition:psistar}) in the last equality. Under C\ref{hyp:C1}\ref{hyp:C1:b}-\ref{hyp:C1:c} and C\ref{hyp:C2},
Lemmas~\ref{lem:Hurwitz:limit} and \ref{lem:Hurwitz:aver} yield for any fixed
$\ell \geq 1$
\begin{equation}
  \label{eq:auxsurpsistar}
  \lim\sup_{n}\sum_{k=2}^{n}\gamma_{k} | \psi_{\star}(n+1, k+1)| < \infty \eqsp,
\qquad \qquad \lim_n|\psi_{\star}(n+1, \ell) | =0\eqsp.
\end{equation}
Under C\ref{hyp:C3}\ref{hyp:C3:a}, for any $\ell \geq 1$, $\PE\left[|S_\ell|^2
\right] \leq \sum_{k\geq \ell}\gamma_{k}^2 \ \PE\left[|e_{k}|^2
  \un_{\A_{m,k-1}} \right]$.  By C\ref{hyp:C2} and C\ref{hyp:C3}\ref{hyp:C3:b},
the rhs is finite for any $\ell \geq 1$, thus implying that (a) $S_\ell$ is
finite w.p.1.  and (b) $\lim_n S_n = 0$ w.p.1.  (\ref{eq:Abelsurmu1}),
(\ref{eq:auxsurpsistar}) and these properties of $S_n$ imply that $\mu_n^{(1)}
\un_{\A_m} \aslim 0$ when $n \to \infty$.

Let us now consider $\mu_n^{(2)}$. By C\ref{hyp:C3}\ref{hyp:C3:b}, there exists
a random index $K$ such that for any $k \geq K$, $(1-\un_{\A_{m,k}}) \un_{\A_m}
\un_{\lim_q \theta_q = \theta_\star} =0$. Hence, for any $n \geq K$,
\begin{equation}
  \label{eq:munTronque}
  \mu_{n+1}^{(2)} \un_{\A_m} \un_{\lim_q \theta_q = \theta_\star} = \sum_{k=1}^{K} \gamma_k \psi_{\star}(n+1,k+1)e_{k}
 \left(1-\un_{\A_{m,k-1}} \right) \un_{\A_m} \un_{\lim_q \theta_q = \theta_\star} \eqsp.
\end{equation}
Then, by (\ref{eq:auxsurpsistar}), $\mu_n^{(2)}\un_{\A_m} \un_{\lim_q \theta_q
  = \theta_\star} = o_{w.p.1.}(1)$.  This concludes the proof of item
\textit{(i)}.

\textit{(ii)} Under C\ref{hyp:C3}\ref{hyp:C3:a}, (\ref{eq:definition:mu})
implies
\[
\PE \left[|\mu_{n+1}^{(1)}|^{2}
 \right] \leq \sum_{k=1}^{n+1}\gamma_{k}^{2}\PE\left[|\psi_{\star}(n+1,k+1)e_{k}
  \un_{\A_{m,k-1}}|^{2} \right] \eqsp.
\]
By C\ref{hyp:C1}\ref{hyp:C1:c}, C\ref{hyp:C3}\ref{hyp:C3:b}, C\ref{hyp:C2} and
Lemma~\ref{lem:Hurwitz:limit}, there exist positive constants $C, L'$ such that
\begin{align*}
  \PE\left[|\mu_{n+1}^{(1)}|^{2} \right]& \leq \sum_{k=1}^{n+1}\gamma_{k}^{2} \ 
  | \psi_{\star}(n+1,\ 
  k+1)|^{2} \ \PE\left[|e_{k}|^{2}\un_{\A_{m,k-1}}\right] \\
  & \leq C \sup_{k}\PE
  \left[|e_{k}|^{2}\un_{\A_{m,k-1}}\right]\sum_{k=1}^{n+1}\gamma_{k}^{2}\exp(-2L'\sum_{j=k+1}^{n+1}\gamma_{j})
  \eqsp.
\end{align*}
Therefore, by Lemma~\ref{lem:Hurwitz:aver} and C\ref{hyp:C2},
$\limsup_k\gamma_{k}^{-1}\PE[|\mu_{k}|^{2}]<+\infty$. Consider now
$\mu_{n+1}^{(2)}$. By C\ref{hyp:C2} and Lemma~\ref{lem:Hurwitz:limit}, $\lim_n
\gamma_n^{-1} |\psi_{\star}(n, \ell) |^2 \to 0$ for any fixed $\ell$.
Therefore, by (\ref{eq:munTronque}), $\gamma_n^{-1} |\mu_{n+1}^{(2)}|^{2}
\un_{\A_m} \un_{\lim_q \theta_q = \theta_\star}= o_{w.p.1}(1)$. This concludes
the proof of the second item.
\end{proof}

\subsection{Preliminary results on the sequence $\{\rho_n, n\geq 0\}$}
By (\ref{eq:definition:rho}) and (\ref{eq:Taylor:h}),
\begin{align*}
  \rho_{n+1}&=(\Id+\gamma_{n+1}\nabla
  h(\theta_{\star}))\rho_{n}+ \gamma_{n+1}r_{n+1}+\gamma_{n+1}(\theta_{n}-\theta_{\star})^{T}R^{(n)}_\bullet(\theta_{n}-\theta_{\star}) \\
  & = (\Id+\gamma_{n+1}\nabla h(\theta_{\star}))\rho_{n}  +\gamma_{n+1}r_{n+1}+\gamma_{n+1}(\mu_{n}+\rho_{n})^{T}R^{(n)}_\bullet(\mu_{n}+\rho_{n}) \\
  & = \left(\Id+\gamma_{n+1}\nabla
    h(\theta_{\star})+2\gamma_{n+1}\mu_{n}^{T}R^{(n)}_\bullet
    +\gamma_{n+1}\rho_{n}^{T}R^{(n)}_\bullet \right)\rho_{n} \\
  & \hspace{2cm}
  +\gamma_{n+1}r_{n+1}+\gamma_{n+1}\mu_{n}^{T}R^{(n)}_\bullet\mu_{n}\eqsp.
\end{align*}
By induction, this yields
\begin{equation}
 \label{eq:expression:rho}
 \rho_{n} = \psi(n,1)\rho_{0} + \sum_{k=1}^{n}\gamma_{k}\psi(n,
 k+1)\left(r_{k}+\mu_{k-1}^{T}R^{(k-1)}_\bullet \mu_{k-1} \right) \eqsp,
\end{equation}
where $\psi(n,k)$ is given by (\ref{eq:definition:psi}).

\bigskip 
\begin{prop}
  \label{prop:rho:rate}
  Assume C\ref{hyp:C1}, C\ref{hyp:C3}\ref{hyp:C3:a}-\ref{hyp:C3:b} and
  C\ref{hyp:C2}. Let $\theta_0 \in \Theta$.  Then, for any $m \geq 1$,
  \[
 \left\{ \rho_{n} - \sum_{k=1}^{n}
    \gamma_k \psi(n,k+1) r_k \right\} \un_{\lim_{q}\theta_{q}=\theta_{\star}}
  \un_{\A_m}= \gamma_{n}^{1 \wedge (1/2+\kappa)} \ O_{w.p.1}(1)O_{L^1}(1) \eqsp,
\]
with $ \kappa = 1/2$ under C\ref{hyp:C2}\ref{hyp:C2:a:1} and $\kappa \in (0, L
\gamma_\star -1/2)$ under C\ref{hyp:C2}\ref{hyp:C2:a:2}.

Assume in addition C\ref{hyp:C4}.  Then, for any $m \geq 1$,
\[
\sum_{k=1}^{n} \gamma_k \psi(n,k+1) r_k
\un_{\lim_{q}\theta_{q}=\theta_{\star}} \un_{\A_m}=\gamma_{n}^{1/2} \ O_{w.p.1}(1)o_{L^1}(1) \eqsp.
\]
\end{prop}
\begin{proof} 
  The proof in given under C\ref{hyp:C2}\ref{hyp:C2:a:2}. The case
  C\ref{hyp:C2}\ref{hyp:C2:a:1} - which is simpler - is on the same lines and
  is omitted. Let $m \geq 1$ be fixed.
  
  \textit{(i)} Let $\eta >0$ and $\kappa \in (0, L \gamma_\star - 1/2)$ such
  that
\begin{equation}
  \label{eq:condition:kappa}
  (L-\eta)\gamma_{\star}>  1/2+\kappa \eqsp.
\end{equation}
Note that such $(\eta,\kappa)$ exist under C\ref{hyp:C2}\ref{hyp:C2:a:2}. This
implies that
\begin{equation}
  \label{eq:rho:rate:tool2}
\lim_{n}\sup\gamma_{n}^{-(1/2+\kappa)}\exp(-(L-\eta)\sum_{j=1}^{n}\gamma_{j})<+\infty
\eqsp.
\end{equation}

We now prove by application of Lemma~\ref{lem:Hurwitz:limit} that there exists
an almost-surely finite positive r.v. $U_{\eta}$ such that for any $1\leq k\leq
n$,
\begin{equation}
  \label{eq:rho:rate:tool1}
  | \psi(n, k)| \un_{\lim_{q}\theta_{q}=\theta_{\star}}\un_{\A_m}\leq
U_{\eta}\exp(-(L-\eta)\sum_{j=k}^{n}\gamma_{j}) \eqsp.
\end{equation}
To that goal, let us prove w.p.1.  $\lim_n (
\rho_{n}^{T}R^{(n)}_\bullet+2\mu_{n}^{T}R^{(n)}_\bullet)\un_{\lim_{q}\theta_{q}=\theta_{\star}}\un_{\A_m}=0$.
By Proposition~\ref{prop:mu:ps-L2}, $
\lim_{n}\mu_{n}\un_{\A_m}\un_{\lim_{q}\theta_{q}=\theta_{\star}}=0$ w.p.l. and
this implies that w.p.1.,
\[
\lim_{n}\rho_{n}\un_{\lim_{q}\theta_{q}=\theta_{\star}}\un_{\A_m}=\lim_{n}(\theta_{n}-\theta_{\star}-\mu_{n})\un_{\lim_{q}\theta_{q}=\theta_{\star}}\un_{\A_m}=0
\eqsp.
\]
In addition, under C\ref{hyp:C1}\ref{hyp:C1:b}, $R^{(n)}_\bullet
\un_{\lim_{q}\theta_{q}=\theta_{\star}} = O_{w.p.1.}(1)$. This concludes the
proof of (\ref{eq:rho:rate:tool1}). Set $\kappa' \eqdef 1 \wedge (1/2+\kappa)$.
By (\ref{eq:expression:rho}), we have
\begin{multline}
\label{eq:rho:rate:tool5}
\rho_{n} - \sum_{k=1}^{n}\gamma_{k}\psi(n, k+1)r_{k} = \psi(n,1)\rho_{0} +
\sum_{k=1}^{n}\gamma_{k}\psi(n, k+1)\left(\mu_{k-1}^{T}R^{(k-1)}_\bullet
  \mu_{k-1} \right) \eqsp.
\end{multline}
Consider the first term.  By (\ref{eq:rho:rate:tool1}), 
\[
\gamma_{n}^{-(1/2+\kappa)}| \psi(n,\ 1)| \ 
|\rho_{0}|\un_{\lim_{q}\theta_{q}=\theta_{\star}}\un_{\A_m}\leq\gamma_{n}^{-(1/2+\kappa)}U_{\eta}\exp(-(L-\eta)\sum_{j=1}^{n}\gamma_{j})|\rho_{0}|,
\]
and by (\ref{eq:rho:rate:tool2}), this term is $O_{w.p.1}(1)$.  For the second
term, it holds by (\ref{eq:rho:rate:tool1})
\begin{multline*}
  \gamma_{n}^{-\kappa'} |\sum_{k=1}^{n}\gamma_{k}\psi(n,
  k+1)\mu_{k-1}^{T}R^{(k-1)}_\bullet\mu_{k-1}|\un_{\lim_{q}\theta_{q}=\theta_{\star}} \un_{\A_m} \\
  \leq\gamma_{n}^{-\kappa' } \sum_{k=1}^{n}\gamma_{k}^{2} \ | \psi(n, k+1)|
  \ (\gamma_{k}^{-1/2}|\mu_{k-1}|)^{2} | R^{(k-1)}_\bullet |
  \un_{\lim_{q}\theta_{q}=\theta_{\star}} \un_{\A_m} \\
  \leq O_{w.p.1}(1) \ 
  \gamma_{n}^{-\kappa'}\sum_{k=1}^{n}\gamma_{k}^{1 + \kappa'}\exp(-(L-\eta)\sum_{j=k+1}^{n}\gamma_{j})(\gamma_{k}^{-1/2}|\mu_{k-1}|)^{2}
  | R^{(k-1)}_\bullet | \un_{\lim_{q}\theta_{q}=\theta_{\star}} \un_{\A_m}\eqsp.
\end{multline*}
By Proposition~\ref{prop:mu:ps-L2}, $ \gamma_{n}^{-1}|\mu_{n}|^2
\un_{\A_m}\un_{\lim_{q}\theta_{q}=\theta_{\star}}= O_{L^1}(1) + o_{w.p.1}(1)$
and under C\ref{hyp:C1}\ref{hyp:C1:b}, $| R^{(k)}_\bullet |
\un_{\lim_{q}\theta_{q}=\theta_{\star}} = O_{w.p.1}(1)$.  By
Lemma~\ref{lem:Hurwitz:aver} and (\ref{eq:condition:kappa}), this term is
$O_{w.p.1}(1) O_{L^1}(1)$. 

\textit{(ii)} Set $r_ k = r_k^{(1)} + r_k^{(2)}$ as in C\ref{hyp:C4}. It holds,
\begin{multline*}
  \gamma_{n}^{- 1/2}|\sum_{k=1}^{n}\gamma_{k}\psi(n,\ 
  k+1)r_{k}^{(1)}|\un_{\lim_{q}\theta_{q}=\theta_{\star}}\un_{\A_m} \\
  \leq\gamma_{n}^{-1/2}\sum_{k=1}^{n}\gamma_{k}^{3/2} \ |\psi(n,\ k+1)| \ 
  |\gamma_{k}^{-1/2}r_{k}^{(1)}|\un_{\lim_{q}\theta_{q}=\theta_{\star}} \un_{\A_m}
\end{multline*}
and by (\ref{eq:condition:kappa}), (\ref{eq:rho:rate:tool1}), C\ref{hyp:C4},
C\ref{hyp:C2} and Lemma~\ref{lem:Hurwitz:aver}, this term is
$O_{w.p.1}(1)o_{L^1}(1)$. For the second term, we use the Abel lemma: set
$\Xi_n \eqdef \sum_{k=1}^n r_k^{(2)} \un_{\A_m} \un_{\lim_q \theta_q
  =\theta_\star}$. Then
\begin{multline*}
\un_{\A_m} \un_{\lim_q \theta_q
  =\theta_\star}  \gamma_{n}^{- 1/2}\sum_{k=1}^{n}\gamma_{k}\psi(n,\ 
  k+1)r_{k}^{(2)} \\
  = \sqrt{\gamma_n} \Xi_n + \gamma_{n}^{- 1/2}\sum_{k=1}^{n-1} \gamma_k
  \gamma_{k+1} \psi(n, k+2) \left\{ (\frac{1}{\gamma_{k+1}}-
    \frac{1}{\gamma_k}) I + H_{k} \right\} \Xi_k
\end{multline*}
where $H_{k} = \nabla h(\theta_\star) + 2 \mu_k^T R_\bullet^{(k)} + \rho_k^T
R_\bullet^{(k)}$. Following the same lines as above, along the event $\{\lim_q
\theta_q = \theta_\star\} \cap \A_m$, $\sup_k |H_k|$ is finite w.p.1. Hence, o
\begin{multline*}
 \un_{\A_m} \un_{\lim_q \theta_q
  =\theta_\star} \gamma_{n}^{- 1/2} \left|\sum_{k=1}^{n}\gamma_{k}\psi(n,\ k+1)r_{k}^{(2)}
  \right| \leq O_{w.p.1}(1) \ \left\{ \sqrt{\gamma_n} |\Xi_n|  \right. \\
  \left. + \gamma_{n}^{- 1/2}\sum_{k=1}^{n-1} \sqrt{\gamma_k} \gamma_{k+1}
    |\psi(n, k+2)| \sqrt{\gamma_k}| \Xi_k| \right\}
\end{multline*}
and the rhs is $O_{w.p.1}(1) o_{L^1}(1)$ by Lemma~\ref{lem:Hurwitz:aver} and
C\ref{hyp:C4}.
\end{proof}

\subsection{Proof of Theorem~\ref{theo:clt}}
\label{sec:proof:theo:clt}
By (\ref{eq:definition:rho}), $\gamma_n^{-1/2} \left( \theta_n - \theta_\star
\right) = \gamma_n^{-1/2} \mu_n + \gamma_n^{-1/2} \rho_n$.  We first prove that on $\{\lim_q \theta_q = \theta_\star \}$,
the second term tends to zero in probability. By C\ref{hyp:C3}\ref{hyp:C3:b},
for any $\epsilon >0$ there exists $m \geq 1$ such that $\PP(\A_m \vert \lim_q
\theta_q = \theta_\star) \geq 1 - \epsilon$.  Therefore, it is sufficient to
prove that for any $m \geq 1$, $\gamma_n^{-1/2} \rho_n \un_{\A_m} \un_{\lim_q
  \theta_q = \theta_\star} \plim 0$ when $n \to \infty$. This property holds by
Proposition~\ref{prop:rho:rate}.

\bigskip

We now prove a CLT for the sequence $\{\gamma_n^{-1/2} \mu_n, n\geq 0\}$.  It
is readily seen that
\[
\lim_n \PE\left[\exp(i \gamma_n^{-1/2} t^T\mu_n) \un_{\lim_q \theta_q =
    \theta_\star} \right] = \PE\left[\exp(-0.5 t^T V t) \un_{\lim_q \theta_q
    =\theta_\star} \right]
\]
if and only if
\[
\lim_n \PE\left[\exp(i  \gamma_n^{-1/2} t^T \mu_n\un_{\lim_q \theta_q =
    \theta_\star} ) \right] = \PE\left[\exp(-0.5 t^T V t\un_{\lim_q \theta_q =
    \theta_\star} )  \right]
\]
Furthermore, by C\ref{hyp:C2} and Lemma~\ref{lem:Hurwitz:limit}, for any fixed
$\ell \geq 1$, $\lim_n \gamma_n^{-1/2} |\psi_\star(n,\ell)| =0$ (where
$\psi_\star$ is given by (\ref{eq:definition:psistar})); this property,
together with (\ref{eq:expression:mu}) and (\ref{eq:definition:Ansequence})
imply that
\[
\lim_n \PE\left[\exp(i \gamma_n^{-1/2} t^T\mu_n\un_{\lim_q \theta_q =
    \theta_\star} ) \right] = \lim_n \PE\left[ \exp\left( i t^T \sum_{k=1}^n
    X_{n+1,k}  \un_{A_{k-1}}\right) \right]
\]
where $ X_{n+1,k} = \gamma_{n+1}^{-1/2}\gamma_{k}\psi_{\star}(n+1,\ k+1)e_{k}$.
By C\ref{hyp:C3}\ref{hyp:C3:a} and (\ref{eq:definition:Ansequence}),
$\PE\left[X_{n+1,k} \un_{A_{k-1}} \vert \F_{k-1} \right]=0$ and the limit in
distribution is obtained by standard results on CLT for martingale-arrays (see
e.g.  Hall and Heyde~\cite[Corollary 3.1.]{hall:heyde:1980}).

\paragraph{ Lindeberg condition} we have to prove that for any $\epsilon>0$, 
\[
\sum_{k=1}^n \PE\left[ |X_{n+1,k}|^2 \un_{|X_{n+1,k} | \geq \epsilon} \ \vert
  \F_{k-1} \right]  \un_{A_{k-1}}\plim 0 \eqsp.
\] 
Following the same lines as above, it can be proved that equivalently, we have
to prove  for any $m \geq 1$,
\[
\un_{\A_m}\un_{\lim_q \theta_q = \theta_\star}  \sum_{k=1}^n \PE\left[ |X_{n+1,k}|^2 \un_{|X_{n+1,k} | \geq \epsilon} \ \vert
\F_{k-1} \right] \plim 0 \eqsp.
\] 
Let $m \geq 1$ be fixed and set $X_{n+1,k} = X_{n+1,k}^{(1)} +
X_{n+1,k}^{(2)}$ with
\[
X_{n+1,k}^{(1)} = X_{n+1,k} \un_{\A_{m,k-1}} \eqsp, \qquad X_{n+1,k}^{(2)} =
X_{n+1,k} \left(1 - \un_{\A_{m,k-1}} \right) \eqsp.
\]
We can assume without loss of generality that $\tau$ given by
C\ref{hyp:C3}\ref{hyp:C3:b} is small enough so that
$(2+\tau)L\gamma_{\star}>1+\tau$. Then,
\begin{multline*}
  \sum_{k=1}^{n+1} \PE\left[|X_{n+1,k}^{(1)}|^{2 +\tau} \right]= \sum_{k=1}^{n+1}\PE
  \left[|\gamma_{n+1}^{-1/2}\gamma_{k}\psi_{\star}(n+1,k+1)
    e_{k} \un_{\A_{m,k-1}}|^{2+\tau} \right] \\
  \leq\sup_{k}\PE\left[|e_k
    \un_{\A_{m,k-1}}|^{2+\tau}\right] \ 
  \gamma_{n+1}^{-1-\tau/2}\sum_{k=1}^{n+1}\gamma_{k}^{2+\tau}|\psi_{\star}(n+1,
  k+1)|^{2+\tau} \eqsp.
\end{multline*}
Under C\ref{hyp:C1}\ref{hyp:C1:b}-\ref{hyp:C1:c}, 
C\ref{hyp:C3}\ref{hyp:C3:b} and C\ref{hyp:C2}, Lemmas~\ref{lem:Hurwitz:limit} and
\ref{lem:Hurwitz:aver} imply
\[
\limsup_{n}\gamma_{n+1}^{-(1+\tau)}\sum_{k=1}^{n+1}\gamma_{k}^{2+\tau}|\psi_{\star}(n+
1,k+1)|^{2+\tau}<+\infty 
\]  since $(2+\tau)L\gamma_{\star}>1+\tau$,
Lemma~\ref{lem:Hurwitz:aver} applies even in the case
C\ref{hyp:C2}\ref{hyp:C2:a:2}). Hence,
\[
 \sum_{k=1}^{n+1} \PE\left[|X_{n+1,k}^{(1)}|^{2 +\tau} \right]= o(\gamma_n^{\tau/2}) \eqsp.
\]
Consider now $X_{n+1,k}^{(2)}$. Since there exists a random variable $K$ such
that $\un_{\A_m} (1-\un_{\A_{m,k-1}}) \un_{\lim_q \theta_q = \theta_\star}=0$
for any $k \geq K$, it holds for any $n \geq K$,
\begin{align*}
& \un_{\lim_q \theta_q = \theta_\star}  \un_{\A_m} \sum_{k=1}^n \PE\left[ |X_{n+1,k}^{(2)}|^2 \un_{|X_{n+1,k} | \geq
      \epsilon}\ \vert
    \F_{k-1} \right]  \\
  & = \un_{\lim_q \theta_q = \theta_\star}\un_{\A_m} \sum_{k=1}^K \PE\left[ |X_{n+1,k}|^2 \un_{|X_{n+1,k} |
      \geq \epsilon} \ \vert
    \F_{k-1} \right] \left( 1- \un_{\A_{m,k-1}} \right) \\
  & \leq \un_{\lim_q \theta_q = \theta_\star}\un_{\A_m} \gamma_n^{-1} \sum_{k=1}^K \gamma_k^2
  |\psi_{\star}(n+1,k+1) |^2\PE\left[ |e_k|^2 \ \vert \F_{k-1} \right] \left(
    1- \un_{\A_{m,k-1}} \right) \eqsp.
\end{align*}
Under C\ref{hyp:C2}, this term is $o_{w.p.1}(1)$.  Therefore, the first
condition of \cite[Corollary 3.1.]{hall:heyde:1980} is satisfied.
\paragraph{ Limiting variance} We prove the second condition of \cite[Corollary
3.1.]{hall:heyde:1980}. Set
\begin{align*}
  V_{n}^{(1)} & \eqdef
  \gamma_{n}^{-1}\sum_{k=1}^{n}\gamma_{k}^{2}\psi_{\star}(n,\ 
  k+1)U_{\star}\psi_{\star}(n,\ k+1)^{T}  \un_{\lim_q \theta_q = \theta_\star}\eqsp,   \\
  \overline{V}_{n}^{(2)} & \eqdef
  \gamma_{n}^{-1}\sum_{k=1}^{n}\gamma_{k}^{2}\psi_{\star}(n,\ 
  k+1) \cdots \\
  & \times  \left(\PE[e_{k}e_{k}^{T}|\mathcal{F}_{k-1}] \un_{A_{k-1}} -U_{\star}
    \un_{\lim_q \theta_q = \theta_\star} \right) \ \psi_{\star}(n,\ k+1)^{T}
  \eqsp;
\end{align*}
We prove that $V_{n}^{(1)} \plim V \un_{\lim_q \theta_q = \theta_\star}$ and
$\overline{V}_{n}^{(2)} \plim 0$.  It holds on $\{\lim_q \theta_q =
\theta_\star \}$,
\begin{align*}
  V_{n+1}^{(1)} & = \gamma_{n+1}U_{\star}+\frac{\gamma_{n}}{\gamma_{n+1}}
  \left(\Id+\gamma_{n+1}\nabla h(\theta_{\star})\right) \ V_{n}^{(1)} \ 
  \left(\Id+\gamma_{n+1}\nabla
    h(\theta_{\star})\right)^{T} \\
  & =V_{n}^{(1)}+\gamma_{n}(U_{\star}+\nabla
  h(\theta_{\star})V_{n}^{(1)}+V_{n}^{(1)}\nabla
  h(\theta_{\star})^{T})+\frac{\gamma_{n}-\gamma_{n+1}}{\gamma_{n+1}}V_{n}^{(1)} \\
  & \hspace{2cm}
  +(\gamma_{n+1}-\gamma_{n})U_{\star}+\gamma_{n}\gamma_{n+1}\nabla
  h(\theta_{\star})V_{n}^{(1)}\nabla h(\theta_{\star})^{T}
\end{align*}
and by Lemma~\ref{lem:LyapunovLimit}, $\lim_n V_n^{(1)} =V \un_{\lim_q \theta_q
  = \theta_\star}$ almost-surely. Following the same lines as above, it can be
proved that $\overline{V}_n^{(2)}$ and $V_n^{(2)}$ given by
\[
V_{n}^{(2)} =\un_{\lim_q \theta_q = \theta_\star} \ 
\gamma_{n}^{-1}\sum_{k=1}^{n}\gamma_{k}^{2}\psi_{\star}(n,\ k+1)
\left(\PE[e_{k}e_{k}^{T}|\mathcal{F}_{k-1}] -U_{\star} \right) \ 
\psi_{\star}(n,\ k+1)^{T}
\]
have the same limit in probability.  By C\ref{hyp:C3}\ref{hyp:C3:c}, we write
$V_{n}^{(2)} = \left(V_{n}^{(2,a)} + V_{n}^{(2,b)}\right)\un_{\lim_q \theta_q =
  \theta_\star}$ with
\begin{align*}
  \hspace{-1cm} V_{n}^{(2,a)} & =
  \gamma_{n}^{-1}\sum_{k=1}^{n}\gamma_{k}^{2}\psi_{\star}(n,\ k+1)
  D_{k-1}^{(1)} \psi_{\star}(n,\ k+1)^{T} \\
  V_{n}^{(2,b)} &=
  \gamma_{n}^{-1}\sum_{k=1}^{n}\gamma_{k}^{2}\psi_{\star}(n,\ k+1)
  D_{k-1}^{(2)} \psi_{\star}(n,\ k+1)^{T} \eqsp.
\end{align*}
We have $ \left|V_{n}^{(2,a)} \right|\leq
\gamma_{n}^{-1}\sum_{k=1}^{n}\gamma_{k}^{2} \left|\psi_{\star}(n,\ k+1)
\right|^2 \ |D_{k-1}^{(1)}|$.  By Lemma~\ref{lem:Hurwitz:aver}, there exists a
constant $C$ such that on $\{\lim_q \theta_q = \theta_\star \}$
\[
\limsup_n \left|V_{n}^{(2,a)} \right| \leq C \ \ \limsup_k \left|
  D_{k}^{(1)}\right| \eqsp,\] where we used (\ref{eq:definition:Ansequence}).
The rhs tends to zero w.p.1.  by C\ref{hyp:C3}\ref{hyp:C3:c}. We now consider
$V_{n}^{(2,b)}$.  Since $\lim_m \PP(\A_m \vert \lim_q \theta_q =
\theta_\star) =1$, it is sufficient to prove that for any $m \geq 1$,
$V_{n}^{(2,b)} \un_{\lim_q \theta_q = \theta_\star} \un_{\A_m} \plim 0$ when
$n \to \infty$. Let $m \geq 1$. Set
\[
\Xi_n \eqdef \sum_{j=0}^n D_j^{(2)} \un_{\lim_q \theta_q = \theta_\star}
\un_{\A_m}\eqsp.
\]
By the Abel transform, we have
\begin{multline*}
  V_{n+1}^{(2,b)} \un_{\A_m} \un_{\lim_q \theta_q = \theta_\star} =
  \gamma_{n+1} \Xi_n + \gamma_{n+1}^{-1} \sum_{k=0}^{n-1} \{
  \gamma_{k+1}^2 \psi_\star(n+1,k+2) \Xi_k \psi_\star(n+1,k+2)^T  \\
  - \gamma_{k+2}^2 \psi_\star(n+1,k+3) \Xi_k \psi_\star(n+1,k+3)^T \}
\end{multline*}
Under C\ref{hyp:C3}\ref{hyp:C3:c}, $ \gamma_n \Xi_n \plim 0$. For the second
term, following the same lines as in Delyon~\cite[Proof of Theorem 24, Chapter
4]{delyon:2000}, it can be proved that the expectation of the second term is
upper bounded by
\[
C \ \gamma_{n+1}^{-1} \sum_{k=0}^{n-1} \gamma_{k+1}^2 \ \left|
  \psi_\star(n+1,k+2)\right|^2 \ \left( \gamma_k \PE\left[| \Xi_k | \right]
\right) \eqsp.
\]
Since $\lim_k \gamma_k \PE\left[| \Xi_k | \right] =0$,
Lemma~\ref{lem:Hurwitz:aver} implies that $V_{n}^{(2,b)}\un_{\A_m} \un_{\lim_q \theta_q = \theta_\star} \plim 0$.
This concludes the proof.

\begin{remark} \label{rem:relaxedD2}
  From the proof above, it can be seen that the assumption on the r.v.
  $D_n^{(2)}$ can be relaxed in
\[
\lim_n \gamma_n \PE\left[ \right| \sum_{k=1}^n D_k^{(2)} \un_{A_k} \un_{\A_{m,k}}
\left| \right] = 0 \eqsp.
\]
Observe indeed that in probability,
\[
\lim_n V_n^{(2,b)} \un_{\A_m} \un_{\lim_q \theta_q = \theta_\star} = \lim_n
\gamma_{n}^{-1}\sum_{k=1}^{n}\gamma_{k}^{2}\psi_{\star}(n,\ k+1) D_{k-1}^{(2)}
\psi_{\star}(n,\ k+1)^{T} \un_{\A_{m,k-1}} \un_{A_{k-1}} \eqsp.
\]
\end{remark}

\subsection{Proof of Proposition~\ref{prop:L2norm:theta}}
The proof is prefaced with a preliminary lemma.
\begin{lemma}
  \label{lem:averaging:2}
  Let $\{\gamma_n, n\geq 1\}$ is a (deterministic) positive sequence satisfying
  C\ref{hyp:C2}\ref{hyp:C2:a:1} and $A$ be a (deterministic) $d \times d$
  Hurwitz matrix. Let $\{x_n, n\geq 0\}$ be a sequence of $\Rset^d$-valued r.v.
  satisfying
\[
x_{n+1} = x_n + \gamma_{n+1} A x_n + \gamma_{n+1} \zeta_{n+1}^{(1)} +
\gamma_{n+1} \zeta_{n+1}^{(2)} \eqsp, \qquad n\geq 0 \eqsp,
\]
where
\begin{align*}
  & \sum_{k=1}^n \gamma_k \left( \prod_{j=k+1}^{n+1}
    \left(\Id + \gamma_{j} A \right) \right) \ \zeta_k^{(1)}\un_{\lim_q x_q =0} = \sqrt{\gamma_n} O_{w.p.1}(1) O_{L^2}(1) \eqsp,  \\
  & |\zeta_{n}^{(2)}| \un_{\lim_q x_q =0} = |x_n|^2 \ O_{w.p.1.}(1) \eqsp.
\end{align*}
Then
\[
\gamma_n^{-1} \ |x_n|^2 \un_{\lim_q x_q =0} = O_{w.p.1.}(1) O_{L^1}(1) \eqsp.
\]
\end{lemma}
\begin{proof}
  The proof is adapted from Delyon~\cite[Theorems 20 and 23]{delyon:2000}.  For $n
  \geq 0$, set $x_n \un_{\lim_q x_q =0} = y_n + z_n$ where
  \begin{equation}
    \label{eq:averaging:2:tool1}
    y_{n+1} = \left(\Id + \gamma_{n+1} A \right) y_n +
\gamma_{n+1}\zeta_{n+1}^{(1)}\un_{\lim_q x_q =0} \eqsp, \qquad n \geq 0 \eqsp,
  \end{equation}  
  and $y_0=0$.  The first step of the proof is to show
  \begin{equation}
    \label{eq:eq:averaging:2:tool2}
y_n = \sqrt{\gamma_n} O_{w.p.1}(1) O_{L^2}(1) \eqsp, \qquad z_n = \gamma_n O_{w.p.1}(1) O_{L^1}(1) \eqsp.
  \end{equation}
Then, upon noting that $(y+z)^2 \leq y^2 + 2 (y+z) z$, we write
\[
|x_n|^2 \un_{\lim_q x_q =0} \leq |y_n|^2 + 2 |x_n| |z_n| \un_{\lim_q x_q =0}
\leq \gamma_n O_{L^1}(1) + 2 \gamma_n O_{w.p.1}(1) \ O_{L^1}(1)
\]
since $ |x_n| \un_{\lim_q x_q =0} = O_{w.p.1.}(1)$.  This concludes the proof
of the Lemma. We turn to the proof of (\ref{eq:eq:averaging:2:tool2}).  By
iterating (\ref{eq:averaging:2:tool1}), we have
\[
y_{n+1} = \sum_{k=1}^{n+1} \gamma_k \left\{ \prod_{j=k+1}^{n+1} \left(\Id +
    \gamma_{j} A \right) \right\} \ \zeta_k^{(1)} \un_{\lim_q x_q =0} \eqsp.
\]
Lemmas~\ref{lem:Hurwitz:limit} and \ref{lem:Hurwitz:aver} imply that $y_n =
\sqrt{\gamma_n} O_{w.p.1}(1) O_{L^2}(1)$. It holds
\begin{align*}
  z_{n+1} &= x_{n+1} \un_{\lim_q x_q =0}- y_{n+1} \\
& = \left(\Id + \gamma_{n+1} A \right) (x_n \un_{\lim_q x_q =0}-y_n)
  +\gamma_{n+1} \zeta_{n+1}^{(2)} \un_{\lim_q x_q =0} \\
& = \left(\Id + \gamma_{n+1} A \right) z_n +
  \gamma_{n+1} \zeta_{n+1}^{(2)} \un_{\lim_q x_q =0} \eqsp.
\end{align*}
Under the stated assumptions, Lemmas~\ref{lem:Hurwitz:limit} and
\ref{lem:Hurwitz:aver} imply that $z_n = o_{w.p.1}(1)$. We thus also have $ y_n
= x_n \un_{\lim_q x_q =0} - z_n = o_{w.p.1}(1)$.  In addition,
\[
|z_{n+1}| \leq | \Id + \gamma_{n+1} A | \ |z_n| + \gamma_{n+1} \ |
\zeta_{n+1}^{(2)} |\un_{\lim_q x_q =0} \eqsp,
\]
and since $A$ is a Hurwitz matrix, there exists a constant $L'>0$ such that $|
\Id + \gamma_{n+1} A | \leq \exp(-L' \gamma_{n+1})$ (see
Lemma~\ref{lem:Hurwitz:limit}). Hence,
\begin{multline*}
  |z_{n+1}|  \leq \exp(-L' \gamma_{n+1}) |z_n| + O_{w.p.1}(1) \ \gamma_{n+1}
  \left( |y_n|^2 + |z_n|^2 \right)   \\
   \leq \exp(- L'\gamma_{n+1}) \left\{ 1 + O_{w.p.1}(1) \ \exp(L' \gamma_{n+1})
    \gamma_{n+1} |z_n| \right\} \ |z_n| + O_{w.p.1}(1) \ \gamma_{n+1} |y_n|^2
  \eqsp.
\end{multline*}
Let $\delta \in(0, L')$. Since $z_n = o_{w.p.1}(1)$, there exists a r.v. $K$
which is finite w.p.1. such that for any $k \geq K$, $|O_{w.p.1}(1) \exp(L'
\gamma_{k+1}) z_k| \leq \delta$.  Therefore, upon noting that for any $x \geq
0$, $1+ x \leq \exp(x)$, for any $n \geq K$,
\begin{align*}
  |z_{n+1}| & \leq \exp(-(L'-\delta) \gamma_{n+1}) |z_n| +  O_{w.p.1}(1) \ \gamma_{n+1} |y_n|^2 \\
  & \leq \exp\left( -(L'-\delta) \sum_{k=K+1}^{n+1} \gamma_k \right) |z_K|  \\
  & \qquad + O_{w.p.1}(1) \ \sum_{k=K+1}^{n+1} \gamma_k \exp\left( -(L'-\delta)
    \sum_{j=k+1}^{n+1} \gamma_k \right) |y_{k-1}|^2 \\
  & \leq O_{w.p.1}(1) \  \exp\left( -(L'-\delta) \sum_{k=1}^{n+1} \gamma_k \right) \\
  & \qquad + O_{w.p.1}(1) \ \sum_{k=1}^{n+1} \gamma_k \exp\left( -(L'-\delta)
    \sum_{j=k+1}^{n+1} \gamma_k \right) |y_{k-1}|^2 \\
  & \qquad  + O_{w.p.1}(1) \ \exp\left( -(L'-\delta) \sum_{j=K+1}^{n+1} \gamma_k \right)
  \eqsp.
\end{align*}
Since $y_n = \sqrt{\gamma_n} O_{L^2}(1)$, C\ref{hyp:C2}\ref{hyp:C2:a:1} and
Lemma~\ref{lem:Hurwitz:aver} imply that $z_n = \gamma_n O_{w.p.1}(1) \ 
O_{L^1}(1)$.
\end{proof}

\paragraph{Proof of Proposition~\ref{prop:L2norm:theta}}
By (\ref{eq:Taylor:h})
\begin{multline*}
  \theta_{n+1} - \theta_\star = \theta_n -\theta_\star + \gamma_{n+1} \nabla
  h(\theta_\star) \ \left( \theta_n -\theta_\star \right) \\
  + \gamma_{n+1} \left(e_{n+1} + r_{n+1} \right) + \gamma_{n+1} \left( \theta_n
    -\theta_\star \right)^T R^{(n)}_\bullet \left( \theta_n -\theta_\star
  \right)
\end{multline*}
Let $m \geq 1$. We apply Lemma~\ref{lem:averaging:2} with $x_n \leftarrow
(\theta_n - \theta_\star) \un_{\A_m}$, $A \leftarrow \nabla h(\theta_\star)$,
$\zeta_{n+1}^{(1)} = (e_{n+1} + r_{n+1}) \un_{\A_m}$ and $\zeta_{n+1}^{(2)} =
\left( \theta_n -\theta_\star \right)^T R^{(n)}_\bullet \left( \theta_n
  -\theta_\star \right) \un_{\A_m}$. Under C\ref{hyp:C1}\ref{hyp:C1:c}, $A$ is a Hurwitz
matrix and $|\zeta_{n+1}^{(2)}| \un_{\lim_q \theta_q = \theta_\star} =
O_{w.p.1}(1) \ |x_n |^2$.

We write $ \zeta_{n+1}^{(1)} = \left( e_{n+1} \un_{\A_{m,n}} + e_{n+1} \left(1
    - \un_{\A_{m,n}}\right) + r_{n+1} \right) \un_{\A_m}$.  Under
C\ref{hyp:C2}, AVER\ref{hyp:Aver3}\ref{hyp:Aver3:a}-\ref{hyp:Aver3:b}, Lemmas
\ref{lem:Hurwitz:limit} and \ref{lem:Hurwitz:aver} imply
\[
\sum_{k=1}^n \gamma_k \psi_\star(n+1,k+1) \ e_{k} \un_{\A_{m,k-1}} =
\sqrt{\gamma_n} O_{L^2}(1) \eqsp.
\]
Upon noting that $\un_{\A_m} \left(1- \un_{\A_{m,k}} \right) = 0$ for all $k
\geq K$ where $K$ is a r.v. finite w.p.1.
\begin{multline*}
  \left(\sum_{k=1}^n \gamma_k \psi_\star(n+1,k+1) \ e_{k} \left(1-
      \un_{\A_{m,k-1}} \right) \right)
  \un_{\A_m} \\
  = \left(\sum_{k=1}^K \gamma_k \psi_\star(n+1,k+1) \ e_{k} \left(1-
      \un_{\A_{m,k-1}} \right) \right) \un_{\A_m} \eqsp.
\end{multline*}
Therefore, by Lemma~\ref{lem:Hurwitz:aver}, this second term is
$\sqrt{\gamma_n} O_{w.p.1}(1)$. Finally, Lemma~\ref{lem:Hurwitz:aver} and
AVER\ref{hyp:Aver4}\ref{hyp:Aver4:a}-\ref{hyp:Aver4:a'} imply that the last
term is $\sqrt{\gamma_n} O_{w.p.1}(1) O_{L^2}(1)$ (the proof is on the same
lines as the proof of Proposition~\ref{prop:rho:rate} and details are omitted).

\subsection{Proof of Theorem~\ref{theo:clt:aver}}
The proof is adapted from the proof of Delyon~\cite[Theorem 26]{delyon:2000}.  Under
C\ref{hyp:C1}\ref{hyp:C1:c}, $\nabla h(\theta_\star)$ is invertible.  By
(\ref{eq:definition:theta}) and Lemma~\ref{lem:averaging} applied with $x_k
\leftarrow \theta_k - \theta_\star$ and $A \leftarrow \nabla h(\theta_\star)$,
we have 
\[
\sqrt{n} \left(\btheta_n - \theta_{\star} \right) = -\nabla
h(\theta_\star)^{-1} \frac{\sqrt{n}}{n+1} \sum_{k=0}^n e_{k+1} + \sqrt{n} Z_n
\]
where 
\begin{align*}
  \nabla h(\theta_\star) Z_n & \eqdef - \frac{1}{n+1} \sum_{k=0}^n r_{k+1} -
  \frac{1}{n+1} \sum_{k=0}^n \left( h(\theta_k) - \nabla h
    (\theta_\star) (\theta_k -\theta_\star)  \right)\\
  &\hspace{-1cm} + \frac{1}{n+1}
  \left(\frac{\theta_{n+1}-\theta_\star}{\gamma_{n+1}} -
    \frac{\theta_0-\theta_\star}{\gamma_1} \right) + \frac{1}{n+1} \sum_{k=1}^n
  \left( \frac{1}{\gamma_k} -\frac{1}{\gamma_{k+1}}\right) \left( \theta_k -
    \theta_\star \right) \eqsp.
\end{align*}
We prove that $\sqrt{n} Z_n \un_{\lim_q \theta_q =\theta_\star} \plim 0$;
combined with AVER\ref{hyp:Aver3}\ref{hyp:Aver3:c}, this will conclude the
proof. Since $\lim_m \PP(\A_m \vert \lim_q \theta_q = \theta_\star)=1$, it is
sufficient to prove that for any $m \geq 1$, $\sqrt{n} Z_n \un_{\A_m}
\un_{\lim_q \theta_q =\theta_\star} \plim 0$. Let $m \geq 1$. By
AVER\ref{hyp:Aver4}\ref{hyp:Aver4:b}, it holds $n^{-1/2} \sum_{k=0}^n r_{k+1}
\un_{\A_m} \un_{\lim_q \theta_q =\theta_\star} \plim 0$.  By
(\ref{eq:Taylor:h}),
\[
\frac{1}{n+1} \sum_{k=0}^n \left( h(\theta_k) - \nabla h (\theta_\star)
  (\theta_k -\theta_\star) \right)= \frac{1}{n+1} \sum_{k=0}^n \left( \theta_k
  -\theta_\star\right)^T R^{(k)}_\bullet \left( \theta_k -\theta_\star\right) \eqsp,
\]
and by C\ref{hyp:C1}\ref{hyp:C1:b},  $R^{(k)}_\bullet \un_{\lim_q \theta_q =\theta_\star} = O_{w.p.1}(1)$.
Therefore, by Proposition~\ref{prop:L2norm:theta}, 
\[
  \frac{\sqrt{n}}{n+1} \sum_{k=0}^n \left( h(\theta_k) - \nabla h
    (\theta_\star) (\theta_k -\theta_\star) \right) \un_{\A_m} \un_{\lim_q
    \theta_q =\theta_\star} = \left( \frac{\sqrt{n}}{n+1} \sum_{k=0}^n \gamma_k
    W_k \overline{W}_k \right) \eqsp,
\]
where $W_k = O_{w.p.1.}(1)$ and $\overline{W}_k = O_{L^1}(1)$.
AVER\ref{hyp:Aver2} implies that this term tends to zero in probability.
Proposition~\ref{prop:L2norm:theta} and AVER\ref{hyp:Aver2} imply that
\[
\un_{\A_m} \un_{\lim_q \theta_q =\theta_\star}\frac{\sqrt{n}}{n+1}
\left(\frac{\theta_{n+1}-\theta_\star}{\gamma_{n+1}} -
  \frac{\theta_0-\theta_\star}{\gamma_1} \right) = \frac{O_{L^1}(1)
  O_{w.p.1.}(1)}{\sqrt{(n+1) \gamma_{n+1}} } + o_{w.p.1.}(1) \plim 0 \eqsp.
\]

Finally, Proposition~\ref{prop:L2norm:theta} and
AVER\ref{hyp:Aver2} also imply that
\begin{multline*}
  \un_{\A_m} \un_{\lim_q \theta_q
    =\theta_\star}\frac{\sqrt{n}}{n+1}\sum_{k=1}^n \left( \frac{1}{\gamma_k}
    -\frac{1}{\gamma_{k+1}}\right) \left( \theta_k - \theta_\star \right)  \\
  = \left( \frac{1}{ \sqrt{n}}\sum_{k=1}^n \left| \frac{1}{\gamma_k}
      -\frac{1}{\gamma_{k+1}}\right| \gamma_k^{1/2} W_k \sqrt{\overline{W}_k} \right) 
\end{multline*}
where $W_k = O_{w.p.1.}(1)$ and $\overline{W}_k = O_{L^1}(1)$. This term tends
to zero in probability.

\begin{lemma}
\label{lem:AVER1}
C\ref{hyp:C3} and $\lim_n n \gamma_n >0$ imply AVER\ref{hyp:Aver3}.
\end{lemma}
\begin{proof}
  C\ref{hyp:C3} implies trivially
  AVER\ref{hyp:Aver3}\ref{hyp:Aver3:a}-\ref{hyp:Aver3:b}. We only have to check
  AVER\ref{hyp:Aver3}\ref{hyp:Aver3:c}, or equivalently, prove that for any $m \geq
  1$,
\[
\lim_n \PE\left[ \ \exp\left(i t^T \mathcal{E}_{n+1} \un_{\lim_q \theta_q =
      \theta_\star} \un_{\A_m} \right) \right] =\PE\left[ \ \exp\left(i t^T U_\star t
    \un_{\lim_q \theta_q = \theta_\star} \un_{\A_m}\right) \right] \eqsp.
\]
Write $ \mathcal{E}_{n+1}\un_{\lim_q \theta_q = \theta_\star} \un_{\A_m} =
T_{1,n} + T_{2,n}$ with $ T_{1,n} = (n+1)^{-1/2} \sum_{k=0}^n e_{k+1}
\un_{\A_{m,k}} \un_{A_k} $. By (\ref{eq:definition:Ansequence}) and
C\ref{hyp:C3}\ref{hyp:C3:b}, $T_{2,n} = o_{w.p.1.}(1)$. Observe that
$\PE\left[e_{k+1} \un_{\A_{m,k}} \un_{\A_k} \vert \F_k \right] =0$ so that the
convergence in distribution of $T_{1,n}$ will be established by applying
results on martingale-arrays: we check the assumptions of Hall and
Heyde~\cite[Corollary 3.1.]{hall:heyde:1980}.  By
C\ref{hyp:C3}\ref{hyp:Aver3:b}, it is easily checked that for any $\epsilon
>0$, there exists a constant $C$ such that for any $n$,
\[
\PE\left[ \frac{1}{n} \sum_{k=0}^n \PE\left[|e_{k+1}|^2 \un_{|e_{k+1}| \geq
      \epsilon \sqrt{n}}\vert \F_{k}\right] \un_{\A_{m,k}} \right] \leq
\frac{C}{n^{\tau/2}} \eqsp.
\]
Hence, $ n^{-1} \sum_{k=0}^n \PE\left[|e_{k+1}|^2 \un_{|e_{k+1}| \geq \epsilon
    \sqrt{n}}\vert \F_{k}\right] \un_{\A_{m,k}} \un_{A_k}\plim 0$. We now prove that
\begin{equation}
  \label{eq:tool0:C2toAVER2}
 \frac{1}{n+1} \sum_{k=0}^n \PE\left[ e_{k+1} e_{k+1}^T \vert \F_{k} \right]
\un_{\A_{m,k}}  \un_{A_k}\plim U_\star \un_{\A_m} \un_{\lim_q
  \theta_q = \theta_\star} \eqsp.
\end{equation}
As above, we claim that this is equivalent to the proof that for any $m\geq 1$,
\[
\un_{\lim_q \theta_q = \theta_\star} \un_{\A_m} \frac{1}{n+1} \sum_{k=0}^n
\left( \PE\left[ e_{k+1} e_{k+1}^T \vert \F_{k} \right] - U_\star \right) \plim
0 \eqsp.
\]
C\ref{hyp:C3}\ref{hyp:Aver3:c} and the Cesaro lemma imply that w.p.1, on the
set $\A_m \cap \{\lim_q \theta_q = \theta_\star \}$, $ (n+1)^{-1} \sum_{k=0}^n
D_{k}^{(1)} \aslim 0$. Finally, under C\ref{hyp:C3}\ref{hyp:Aver3:c},
\[
\frac{1}{n+1} \PE\left[ \left| \sum_{k=0}^n D_{k}^{(2)} \un_{\lim_q \theta_q =
      \theta_\star} \un_{\A_m} \right| \right] = \frac{o(1)}{n \gamma_n} 
\]
and the rhs tends to zero since $\lim_n n \gamma_n >0$. This concludes the
proof of (\ref{eq:tool0:C2toAVER2}) and the proof of the Lemma.
\end{proof}

\subsection{Technical lemmas}

\begin{lemma}
  \label{lem:prelim:PierreP}
  Let $(\Omega, \mathcal{A},\mu)$ be a measured space, where $\mu$ is a bounded
  positive measure. Let $\mathcal{G}$ be an algebra generating $\mathcal{A}$.
  Then for all $B \in \mathcal{A}$ and $\epsilon >0$, we can find $A \in
  \mathcal{G}$ such that $\mu(A \Delta B) < \epsilon$.
\end{lemma}
\begin{proof}
  Let $\mathcal{S}\eqdef \{ A \subset \Omega, \forall \epsilon >0, \exists A' \in
  \mathcal{G}, \mu(A \Delta A') \leq \epsilon\}$. We prove that $\mathcal{S}$
  is a $\sigma$-algebra; since it contains $\mathcal{G}$ by definition, this
  yields the result.
  
  $\Omega \in \mathcal{S}$ since $\Omega \in \mathcal{G}$.  Let $A \in
  \mathcal{S}$: we prove that $A^c \eqdef \Omega \setminus A \in \mathcal{S}$.
  Let $\epsilon >0$; there exists $A' \in \mathcal{G}$ such that $\mu(A \Delta
  A') \leq \epsilon$. Since $A \Delta B = A^c \Delta B^c$, it holds
\[
\mu(A^c \Delta (A')^c) = \mu(A \Delta A') \leq \epsilon;
\]
$(A')^c \in \mathcal{G}$ since $\mathcal{G}$ is an algebra, thus showing that
$A^c \in \mathcal{S}$.

Finally, we prove that $\mathcal{S}$ is stable by countable union. We first
prove it is stable by finite union, or equivalently by union of two elements.
Let $A_1, A_2$ be elements of $\mathcal{S}$ and fix $\epsilon >0$. There exists
$A'_k \in \mathcal{G}$ such that $\mu(A_k \Delta A'_k) \leq \epsilon /2$.  Upon
noting that
\begin{equation}
  \label{eq:Property:diffset}
  (A_1 \cup A_2) \Delta (A'_1 \cup A'_2) \subset (A_1 \Delta A'_1) \cup (A_2 \Delta A'_2)
\end{equation}
it holds
\[
\mu\left( (A_1 \cup A_2) \Delta (A'_1 \cup A'_2) \right) \leq \epsilon \eqsp.
\]
This concludes the proof since $A'_1 \cup A'_2 \in \mathcal{G}$. Let us
consider the countable case. Let $(A_k, k \geq 1)$ be a sequence of
$\mathcal{S}$ and fix $\epsilon >0$; since $\mathcal{S}$ is stable under
complement and finite union, we can assume without loss of generality that the
sets $A_k$ are pairwise disjoint.  For any $k$, there exists $A'_k \in
\mathcal{G}$ such that $\mu(A_k \Delta A'_k) \leq \epsilon 2^{-k}$. Since
$(A_k, k \geq 0)$ are pairwise disjoint
\[
\mu\left( \bigcup_k A_k\right) = \sum_{k \geq 1} \mu(A_k) \eqsp;
\]
since $\mu$ is finite, there exists $K_\epsilon$ such that $\mu(\sum_{k >
  K_\epsilon} A_k) \leq \epsilon/2$. Using again (\ref{eq:Property:diffset}) it
holds
\begin{align*}
\mu\left( (\bigcup_k A_k) \Delta (\bigcup_{k \leq K_\epsilon} A'_k)\right)& \leq
\mu\left( (\bigcup_{k \leq K_\epsilon} A_k) \Delta (\bigcup_{k \leq K_\epsilon}
  A'_k)\right) + \mu\left( \bigcup_{k > K_\epsilon} A_k\right)  \\
&\leq  \sum_{k =1}^{K_\epsilon} \mu\left( A_k \Delta A'_k \right) + \epsilon/2  \\
& \leq \epsilon \eqsp.
\end{align*}
Since $\bigcup_{k \leq K_\epsilon} A'_k \in \mathcal{G}$, this concludes the
proof of the sub-additivity.

\end{proof}

\begin{lemma}
  \label{lem:PierreP} Let  $( \Omega, \mathcal{A}, \PP, \{\F_n, n\geq 0\})$ be a filtered probability space and set $\F_\infty = \sigma(\F_n, n\geq 1)$.
  Let $B \in \F_\infty$.  There exists a $\F_n$-adapted sequence $\{A_n,n \geq
  0\}$ such that $\lim_n \un_{A_n} = \un_B$ $\PP$-a.s.
\end{lemma}
  \begin{proof}
    For any $\epsilon >0$, there exist $m \geq 1$ and $\tilde A \in \F_m$ such
    that $\PE\left[| \un_{\tilde A} - \un_B| \right] \leq \epsilon$ (see
    Lemma~\ref{lem:prelim:PierreP}). Therefore, for any $n \geq 1$, there exist
    sets $\tilde A_{n} \in \F_{m_n}$ such that $\PE\left[| \un_{\tilde A_n} -
      \un_B| \right] \leq 1/n$. This implies almost-sure convergence of a
    subsequence $\{\tilde A_{\phi_n}, n\geq 0\}$ to $\un_B$, with $\tilde
    A_{\phi_n} \in \F_{m_{\phi_n}}$.  Note that we can assume without loss of
    generality that the sequence $\{m_{\phi_n} n\geq 1 \}$ is non decreasing.
    For any $k \in [m_{\phi_n}, m_{\phi_{n+1}}[$, set $A_k = \tilde
    A_{\phi_n}$. Then, $A_k \in \F_{m_{\phi_n}} \subseteq \F_k$ and
\[
\lim_k \un_{A_k} = \lim_n \un_{\tilde A_{\phi_n}}= \un_B \eqsp.
\]
  \end{proof}

\begin{lemma}
\label{lem:Hurwitz:limit}
Let $|\cdot|$ be any matrix norm. Let $\{A_{k},\ k\geq 0\}$ be a sequence of
square matrix such that $ \lim_{k}| A_{k}-A|=0$ where $A$ is a Hurwitz matrix.
Denote by $-L$, $L>0$, the largest real part of its eigenvalues. Let
$\{\gamma_{k},\ k\geq 0\}$ be a positive sequence such that $ \lim_k \gamma_k
=0$.  For any $0<L'<L$, there exists a positive constant $C$ such that for any
$k\leq n$
\[
\left| (\mathrm{Id}+\gamma_{n}A_{n}) \cdots (\mathrm{Id}+\gamma_{k+1}A_{k+1})
  (\mathrm{Id}+\gamma_{k}A_{k}) \right|\leq C\exp(- L'\sum_{j=k}^{n}\gamma_{j})
\eqsp.
\]
\end{lemma}
\begin{proof}
 %% The proof can be found in \cite[Exercise 2.I.2]{duflo:1997}.
  Let $\lambda_i, i \leq d$ be the eigenvalues of $A$.  By using the Jordan
  decomposition, we write $A = S J S^{-1}$ where $S$ is a non-singular matrix,
  and $J$ is a Jordan matrix (as defined by Horn and Johnson~\cite[Definition
  3.1.1]{horn:johnson:1985} - note that the diagonal entries of $J$ are
  $\lambda_i$). 
  
  For any $t>0$, denote by $D_t$ the diagonal matrix with diagonal entries
  $(t,t^2, \cdots, t^d)$ and set
\[
A = (S D_t) \left( D_t^{-1} J D_t \right) \left(S D_t \right)^{-1} = (S D_t)
\left( \Lambda + R_t\right) \left(S D_t \right)^{-1}
\]
with $\Lambda \eqdef \mathrm{diag}(\lambda_i)$, upon noting that
\[
D_t^{-1} J D_t = \left[\begin{matrix}
   \lambda_1 & t u_1 & 0 & \cdot & 0 \\
    0 & \lambda_2 & t u_2  & \cdot & 0 \\
\cdot & & & & \cdot \\
0 & \cdot & \cdot & \lambda_{d-1} & t u_{d-1} \\
0 & \cdot & & \cdot & \lambda_d 
\end{matrix} \right]  \eqsp.
\]
Note also that $|R_t| \to 0$ as $t \to 0$.  We write
\begin{align*}
  \left(S D_t \right)^{-1} \left(I + \gamma_\ell A_\ell \right) \left(S D_t
  \right) & = \left(S D_t \right)^{-1} \left(I + \gamma_\ell A \right) \left(S
    D_t \right) + \gamma_\ell \left(S D_t \right)^{-1} \left( A_\ell -A \right)
  \left(S D_t \right)  \\
  & = I + \gamma_\ell \, D_t^{-1} J D_t + \gamma_\ell \left(S D_t \right)^{-1}
  \left( A_\ell -A \right) \left(S D_t \right)  \\
  & = I + \gamma_\ell \, \Lambda + \gamma_\ell R_t + \gamma_\ell \left(S D_t
  \right)^{-1} \left( A_\ell -A \right) \left(S D_t \right) \eqsp.
\end{align*}
Therefore, 
\begin{align*}
  \left| \left(S D_t \right)^{-1} \left(I + \gamma_\ell A_\ell \right) \left(S
      D_t \right)\right| & \leq \left| I + \gamma_\ell \, \Lambda\right| +
  \gamma_\ell \left| R_t \right| + \gamma_\ell \left| A_\ell -A \right| \left|S
    D_t \right| \left| (S D_t)^{-1} \right| \eqsp.
\end{align*}
Let $0 < L < L'' < L$. There exists $t_0$ such that for any $t \in (0, t_0)$,
$|R_t| \leq (L'' -L')/2$; and there exists $K$ such that for any $\ell \geq K$
and any $t \leq t_0$, $\left|I + \gamma_\ell \, \Lambda\right| \leq 1
-\gamma_\ell L''$ and $|A_\ell -A| \left|S D_t \right| \left| (S D_t)^{-1}
\right| \leq (L''-L')/2$. Therefore, for any $\ell \geq K$ and any $t \in
(0,t_0)$
\[
\left| \left(S D_t \right)^{-1} \left(I + \gamma_\ell A_\ell \right) \left(S
    D_t \right)\right| \leq 1 - \gamma_\ell L' \eqsp.
\]
Now we write for $K \leq k<n$ and $t \leq t_0$,
\begin{align*}
  \left|\left(I + \gamma_n A_n\right) \cdots \left(I + \gamma_k A_k\right) =
  \right| \leq \left|S D_t \right| \left| (S D_t)^{-1} \right| \ \prod_{\ell
    =k}^n \left(1 - \gamma_\ell L' \right)
\end{align*}
which concludes the proof.
\end{proof}

\begin{lemma}
\label{lem:Hurwitz:aver}
Let $\gamma_{k}$ be a positive sequence such that $ \lim_{k}\gamma_{k}=0$ and
$\sum_{k}\gamma_{k}=\infty$.  Let $\{e_{k},\ k\geq 0\}$ be a non-negative
sequence. Then
\[
\limsup_n \ \gamma_{n}^{-p}\sum_{k=1}^{n}\gamma_{k}^{p+1} \ e_{k} \ 
\exp(-b\sum_{j=k+1}^{n}\gamma_{j})\leq\frac{1}{C(b,p)}\limsup_n e_{n} \eqsp,
\]
\begin{enumerate}[(i)]
\item with $C(b,p)=b$, for any $b>0,p\geq 0$ if
  $\log(\gamma_{k-1}/\gamma_{k})=o(\gamma_{k})$.
\item with $C(b,p)=b-p/\gamma_{\star}$, for any $b\gamma_{\star}>p\geq 0$ if
  there exists $\gamma_{\star}>0$ such that
  $\log(\gamma_{k-1}/\gamma_{k})\sim\gamma_{k}/\gamma_{\star}$.
\end{enumerate}
By convention, $\sum_{j=n+1}^{n}\gamma_{j}=0$.
\end{lemma}
\begin{proof}
  The proof is from Delyon~\cite[Theorem 19, Chapter 4]{delyon:2000}.  Let $\{x_{n},\ 
  n \geq 0\}$ be defined by
  $x_{n}=\exp(-b\gamma_{n})x_{n-1}+\gamma_{n}^{p+1}e_{n}$ where $x_{0}=0$. Then
  by a trivial recursion, it holds
\[
x_{n}=\sum_{k=1}^{n}\gamma_{k}^{p+1}e_{k}\exp(-b\sum_{j=k+1}^{n}\gamma_{j})
\eqsp.
\]
Set $u_{n} \eqdef \gamma_{n}^{-p}x_{n}$. Then
\begin{align*}
  u_{n} & =\left(\frac{\gamma_{n-1}}{\gamma_{n}}\right)^{p}\exp(-b\gamma_{n})u_{n-1}+\gamma_{n}e_{n} \\
  & =\exp(p\log(\gamma_{n-1}/\gamma_{n})-b\gamma_{n})u_{n-1}+\gamma_{n}e_{n} \\
  & =(1-b_{n}\gamma_{n})u_{n-1}+b_{n}\gamma_{n}(b_{n}^{-1}e_{n}) \eqsp,
\end{align*}
where $b_{n}\gamma_{n} \eqdef
1-\exp(p\log(\gamma_{n-1}/\gamma_{n})-b\gamma_{n})$.  Observe that $\gamma_{n} b_{n}\sim
1-\exp(-b\gamma_{n})$ in case ({\it i}) and $\gamma_{n} b_{n}\sim
1-\exp(-(b-p/\gamma_{\star})\gamma_{n})$ in case ({\it ii}). Therefore,
$\lim_{n} b_{n}=b$ (resp. $b-p/\gamma_{\star}$) in case ({\it i})
(resp. $({\it ii})$).

Let $v \geq \limsup_{n}b_{n}^{-1}e_{n}$. We have
\[
u_{n}-v=(1-b_{n}\gamma_{n})(u_{n-1}-v)+b_{n}\gamma_{n}(b_{n}^{-1}e_{n}-v)
\]
and upon noting that $(a+b)_{+}\leq a_++b_+$, it holds
\[
\left(u_{n}-v\right)_{+}\leq(1-b_{n}\gamma_{n})(u_{n-1}-v)_{+}+b_{n}\gamma_{n}(b_{n}^{-1}e_{n}-v)_{+}\leq(1-b_{n}\gamma_{n})(u_{n-1}-v)_{+} \eqsp.
\]
Since $\lim_n \gamma_n b_n =0$ and $ \sum_{n}b_{n}\gamma_{n}=+\infty,\ 
\lim_{n}(u_{n}-v)_{+}=0$ thus implying that $ \limsup_{n}u_{n}\leq v$. This
holds for any $v \geq \limsup_{n}b_{n}^{-1}e_{n}$ thus concluding the proof.
\end{proof}

\begin{lemma}
\label{lem:matrices}
For any matrices $A, B, C$
\[
\left| ABA^{T}-CBC^{T} \right|= \left|(A-C)BA^{T}-CB(C-A)^{T} \right|\leq
\left| A-C\right| \ \left| B\right| \ (\left| A\right|+\left| C\right|) \eqsp.
\]
\end{lemma}

\begin{lemma}
  \label{lem:LyapunovLimit}
  Let $U_{\star}$ be a positive definite matrix.
  \begin{enumerate}[(a)]
  \item Assume C\ref{hyp:C1}\ref{hyp:C1:b}-\ref{hyp:C1:c} and C\ref{hyp:C2}\ref{hyp:C2:a:1}. Consider the equation
\[
v_{n+1}=v_{n}+
\gamma_{n}f(v_{n})+\frac{\gamma_{n}-\gamma_{n+1}}{\gamma_{n+1}}v_{n}+(\gamma_{n+1}-\gamma_{n})U_{\star}+\gamma_{n}\gamma_{n+1}\nabla
h(\theta_{\star})v_{n}\nabla h(\theta_{\star})^{T} \eqsp,
\]
where $f(v) \eqdef U_{\star}+\nabla h(\theta_{\star})v+v\nabla
h(\theta_{\star})^{T}$. Then there exists an unique positive definite matrix
$V$ such that $f(V)=0$ and $\lim_{n}v_{n}=V$.
\item Assume C\ref{hyp:C1}\ref{hyp:C1:b}-\ref{hyp:C1:c} and
  C\ref{hyp:C2}\ref{hyp:C2:a:2}. Consider the equation
\[
v_{n+1}=v_{n}+\gamma_{n}f(v_{n})+(\gamma_{n+1}-\gamma_{n})U_{\star}+\gamma_{n}\gamma_{n+1}\nabla
h(\theta_{\star})v_{n}\nabla h(\theta_{\star})^{T} \eqsp,
\]
where $f(v) \eqdef U_{\star}+\nabla h(\theta_{\star})v+v\nabla
h(\theta_{\star})^{T}+\gamma_{\star}^{-1}v$. Then there exists an unique
positive definite matrix $V$ such that $f(V)=0$ and $ \lim_{n}v_{n}=V$.
  \end{enumerate}
\end{lemma}
\begin{proof}
\textit{(a)} Let $V$ such that $f(V) =0$. We have
\begin{multline*}
  v_{n+1} -V = v_n-V + \gamma_n\left( H(v_n-V) + (v_n-V) H^T \right) \\ +
  \frac{\gamma_n -\gamma_{n+1}}{\gamma_{n+1}} \left(v_n -V \right) + \gamma_n
  \gamma_{n+1} \nabla h(\theta_\star) (v_n-V) \nabla
  h(\theta_\star)^T  \\
  + \frac{\gamma_n -\gamma_{n+1}}{\gamma_{n+1}} V + \gamma_n \gamma_{n+1}
  \nabla h(\theta_\star) V \nabla h(\theta_\star)^T + \left( \gamma_{n+1} -
    \gamma_n \right) U_\star \eqsp.
\end{multline*}
Under C\ref{hyp:C2}\ref{hyp:C2:a:1},
$|\gamma_{n}/\gamma_{n+1}-1|=o(\gamma_{n})$ and
$\gamma_{n}-\gamma_{n+1}=o(\gamma_{n}^{2})$. Then, denoting by $\bar v_{n}$ the
vectorialized form of the matrix $v_n -V$, this yields
\[
\bar v_{n+1} = \left( \Id + \gamma_n A_n\right) \bar v_n + B_n
\]
where $\{A_n, n\geq 0\}$ is a sequence of Hurwitz matrix that converges to a
Hurwitz matrix $A$, and $B_n = o(\gamma_n) $. Then, there exists $L'>0$ such
that
 \[
 | v_{n+1}-V|\leq(1-\gamma_{n}L') | v_{n}-V|+\gamma_{n}\epsilon_{n} \eqsp,
\]
where $\epsilon_{n}=o(1)$. As in the proof of Lemma~\ref{lem:Hurwitz:aver}, it
can be proved that $ \limsup_{n}| v_{n+1}-V|\leq \limsup_{n}\epsilon_{n}=0$.

\textit{(b)} Under C\ref{hyp:C2}\ref{hyp:C2:a:2},
$\frac{\gamma_{n}-\gamma_{n+1}}{\gamma_{n+1}}= \gamma_{n}/\gamma_{\star} +
o(\gamma_n)$ and $\gamma_{n+1}-\gamma_{n}=O(\gamma_{n}^{2})$.  As in the
previous case, we write
\[
v_{n+1}-V =v_n -V + + \gamma_n \tilde H (v_n-V) + \gamma_n (v_n -V) H^T+  o(\gamma_n)
\]
where $\tilde H \eqdef \nabla h(\theta_\star) + (2 \gamma_\star)^{-1} \Id$.
Under the assumptions on $\gamma_\star$, $\tilde H$ is a Hurwitz matrix.  As in
the proof of Lemma~\ref{lem:Hurwitz:aver}, it can be proved that $\limsup_{n}|
v_{n+1}-V|\leq o(1) =0$.
\end{proof}

\begin{lemma}
\label{lem:averaging}
  Define the sequence $\{x_n, n\geq 0\}$ by
\[
x_{n+1} = x_n + \gamma_{n+1} A x_n + \gamma_{n+1} \zeta_{n+1} \eqsp, \qquad x_0
\in \Rset^d \eqsp,
\]
where $\{\gamma_n, n\geq 1\}$ is a positive sequence, $\{\zeta_n, n\geq 1\}$ is
a $\Rset^d$-valued sequence and $A$ is a $d \times d$ matrix. Then
\begin{align*}
  A \sum_{k=0}^n x_k = - \sum_{k=0}^n \zeta_{k+1} +
  \left(\frac{x_{n+1}}{\gamma_{n+1}} - \frac{x_0}{\gamma_1} \right) +
  \sum_{k=1}^n \left( \frac{1}{\gamma_k} -\frac{1}{\gamma_{k+1}}\right) x_k
  \eqsp.
\end{align*}
\end{lemma}
\begin{proof}
  By definition of $\{x_n, n\geq 0 \}$, for any $n \geq 0$ it holds
\[
A x_n = \frac{1}{\gamma_{n+1}} (x_{n+1} -x_n) - \zeta_{n+1} \eqsp.
\]
Therefore,
\[
A \sum_{k=0}^n x_k = \sum_{k=0}^n \frac{1}{\gamma_{k+1}} (x_{k+1} -x_k)
-\sum_{k=0}^n \zeta_{k+1} \eqsp.
\]
We then conclude by the Abel transform.
\end{proof}

  \begin{lemma} \label{lem:AVER3}
    Let $\{\gamma_n, n \geq 1 \}$ be a positive non-increasing sequence. Then
\[
\lim_n n \gamma_n =+\infty \Longrightarrow \lim_n
\frac{1}{\sqrt{n}}\sum_{k=1}^{n}\gamma_{k}^{-1/2}
\left|1-\frac{\gamma_{k}}{\gamma_{k+1}} \right|=0 \eqsp.
\]
  \end{lemma}
  \begin{proof}
    The following proof can be found in the proof of Delyon~\cite[Theorem 26, Chapter
    4]{delyon:2000}. We have
\begin{multline*}
  \sum_{k=1}^{n}\gamma_{k}^{-1/2} \left(\frac{\gamma_{k}}{\gamma_{k+1}} -1
  \right) =\sum_{k=1}^{n}\gamma_{k}^{1/2} \left(\frac{1}{\gamma_{k+1}}
    -\frac{1}{\gamma_{k}} \right) \\
  = -\sum_{k=2}^{n+1} \gamma_k^{-1} \left( \sqrt{\gamma_k} - \sqrt{\gamma_{k-1}}
  \right) - \frac{1}{\sqrt{\gamma_1}} + \frac{1}{\sqrt{\gamma_{n+1}}} \\
  \leq - \sum_{k=2}^{n+1} \gamma_k^{-1/2} \gamma_{k-1}^{-1/2}\left(
    \sqrt{\gamma_k} - \sqrt{\gamma_{k-1}} \right) +
  \frac{1}{\sqrt{\gamma_{n+1}}} = \frac{1}{\sqrt{\gamma_n}} +
  \frac{1}{\sqrt{\gamma_{n+1}}} \eqsp.
\end{multline*}
  \end{proof}

%% \bibliographystyle{plain}
%% \bibliography{biblioTCL}

\end{document}